\documentclass[a4paper,12pt]{amsart}

\usepackage{amssymb,xspace,amscd,yhmath,mathdots,txfonts,
mathrsfs, youngtab,stmaryrd,float,accents}
\usepackage{array}
\usepackage{dynkin-diagrams}
\usepackage[matrix,arrow,curve]{xy}\CompileMatrices
\usepackage{adjustbox}
\usepackage{graphicx}
\usepackage{hyperref}
\usepackage{changepage}
\usepackage{yfonts}[1998/10/03]
\usepackage{float}

\DeclareMathAlphabet{\mathpzc}{OT1}{pzc}{m}{it}
\numberwithin{equation}{section}

\theoremstyle{plain}

\newtheorem{thm}{Theorem}[section]

\newtheorem{lem}[thm]{Lemma}

\newtheorem{cor}[thm]{Corollary}

\newtheorem{prop}[thm]{Proposition}

\theoremstyle{definition}
\newtheorem{dfn}{Definition}[section]
\newtheorem{ntz}{Notation}[section]

\newtheorem{exam}[thm]{Example}
\newtheorem{nnott}{Notation}[section]

\newtheorem{rmk}[thm]{Remark}

\DeclareMathAlphabet{\mathpzc}{OT1}{pzc}{m}{it}

\DeclareMathOperator{\R}{{\mathbb{R}}}

\DeclareMathOperator{\kt}{{\kappaup}}

\DeclareMathOperator{\Gf}{\mathbf{G}}

\DeclareMathOperator{\supp}{\mathrm{supp}}

\DeclareMathOperator{\Ad}{\mathrm{Ad}}

\DeclareMathOperator{\Z}{\mathbb{Z}}

\DeclareMathOperator{\at}{\mathfrak{a}}
\DeclareMathOperator{\pt}{\mathfrak{p}}
\DeclareMathOperator{\ft}{\mathfrak{f}}

\DeclareMathOperator{\re}{{\mathrm{Re}}}

\DeclareMathOperator{\C}{\mathbb{C}}

\newcommand\Qf{\mathbf{Q}}

\newcommand\qt{\mathfrak{q}}

\newcommand\gt{\mathfrak{g}}

\newcommand\Rad{\mathpzc{R}}

\newcommand\hg{\mathfrak{h}}

\newcommand\Hq{\mathpzc{H}}
\newcommand\Fq{\mathpzc{F}}
\newcommand\Tq{\mathpzc{T}}

\newcommand\spt{\mathfrak{sp}}

\newcommand\Jd{\mathrm{J}}

\newcommand\SU{\mathbf{SU}}

\newcommand\Id{\mathrm{Id}}

\newcommand\Ci{\mathcal{C}^\infty}

\newcommand\slt{\mathfrak{sl}}
\newcommand\su{\mathfrak{su}}

\newcommand\sfE{\textsf{E}}

\newcommand\T{\mathrm{T}}

\newcommand\so{\mathfrak{so}}

\newcommand\Hb{\mathbb{H}}
\newcommand\Qc{\mathpzc{Q}}

\newcommand{\F}{\mathrm{F}}

\newcommand\epi{\varepsilon}

\newcommand\ktt{\texttt{k}}

\newcommand\Ta{\textsf{T}}
\newcommand\Ha{\textsf{H}}
\newcommand\sfM{\textsf{M}}
\newcommand\sfF{\textsf{F}}

\newcommand\e{\texttt{e}}

\newcommand\gs{\mathfrak{g}^{\sigmaup}}

\newcommand\Bz{\mathcal{B}}

\newcommand\Fc{\mathcal{F}}

\newcommand\hs{\mathfrak{h}^{\sigmaup}}

\newcommand\utilde[1]{\underaccent{\tilde}#1}
\newcommand\de{\textswab{d}_{e}}

   \def\DHLhksqrt#1#2{\setbox0=\hbox{$#1\sqrt{#2\,}$}\dimen0=\ht0
     \advance\dimen0-0.2\ht0
     \setbox2=\hbox{\vrule height\ht0 depth -\dimen0}%
     {\box0\lower0.4pt\box2}}
\title{On finitely nondegenerate closed homogeneous 
 CR  manifolds}
\author{S.Marini, C.Medori, M.Nacinovich}
\address{Stefano Marini: Dipartimento di Scienze Matematiche, Fisiche e 
Informatiche\\ Universit\`a di Parma\\ Parco Area delle Scienze 53/A 
(Campus), 43124 Parma
 (Italy)} \email{stefano.marini@unipr.it}

\address{Costantino Medori:
Dipartimento di Scienze Matematiche, Fisiche e Informatiche\\ Universit\`a 
di Parma\\ Parco Area delle Scienze 53/A (Campus), 43124 Parma
 (Italy)} \email{costantino.medori@unipr.it}

\address{Mauro Nacinovich:
Dipartimento di Matematica\\ II Universit\`a di Roma
``Tor Ver\-ga\-ta''\\ Via della Ricerca Scientifica\\ 00133 Roma
(Italy)}
\email{nacinovi@mat.uniroma2.it}


\date\today

\begin{document}

\thanks{
This work has been financially supported by the Programme “FIL-Quota 
Incentivante” of University of Parma, co-sponsored by Fondazione 
Cariparma, by the PRIN project ``Real and Complex Manifolds: Topology, 
Geometry and holomorphic dynamics''- CUP D94I19000800001, and by 
the group G.N.S.A.G.A. of I.N.d.A.M.} 

\maketitle 
\begin{abstract}
A complex flag manifold $\sfF{=}\Gf/\Qf$ decomposes into finitely many  real orbits 
under the action of a real form 
$\Gf^\sigmaup$ of $\Gf$.  Their embedding into $\sfF$ 
define on them CR manifold structures.  We characterize the closed real orbits which are finitely nondegenerate.

\end{abstract}
\tableofcontents

\section*{Introduction}

In order to solve the biholomorphic 
equivalence problem for smooth complex domains,  
one needs to study  biholomorphic invariants of their boundaries.
These are  \emph{closed} smooth hypersurfaces
and are the codimension one 
prototypes of 
the more general objects that
we call now \emph{Cauchy Riemann} (or CR) \emph{manifolds}.   \par

Within specific classes of CR manifolds one aims at determining 
complete systems of geometric invariants characterizing equivalence. 
An important step 
in this context is investigating  
transformation groups 
preserving these invariants.

\par
In this paper we will consider CR manifolds which are orbits of real forms in complex flag manifolds
(cf. e.g. \cite{Wo1969}).\par
Let  $\Gf$ be a complex semisimple Lie group,  $\Qf$ a parabolic subgroup and $\Gf^{\sigmaup}$ a real form of $\Gf$.  
The complex flag manifold $\sfF{=}\Gf/\Qf$ decomposes 
into finitely many $\Gf^{\sigmaup}$-orbits; among them there exists 
exactly one %orbit 
of minimal dimension, which is closed and hence compact 
(see  \cite{AlMeNa2008} for details on its topology).   
Their embedding into $\sfF$ yields natural  CR-manifold structures. 
\par
An important characteristic of a  CR  manifold is the degeneracy/non\-de\-gen\-er\-acy of its Levi form. 
This is  a vector valued 
sesquilinear form which,  roughly speaking, measures the \textit{non-integrability}
of its  CR  distribution.  In several instances this form has a nontrivial kernel and, 
by considering its higher order iterations, we are lead to   
weaker notions of nondegeneracy (see e.g. \cite{Fe2007,Sa2020}). 
This leads to a notion of \emph{Levi order}, 
%We get an \textit{order of  nondegeneracy}, 
expressed by 
 a positive integer $\ktt$ (degeneracy corresponding to $\ktt\,{=}\,{+}\infty$).  
 In \cite {MaMeNa2021} we found the upper bound $\ktt{=}3$ 
 for the finitely nondegenerate orbits of real forms in complex flag manifolds,  
 which drops to $2$ when restricting to the closed ones
(or, more generally, to those of \textit{minimal type}). \par
A fundamental tool, while studying homogeneous  CR  manifolds, 
is employing the notion of CR  algebra introduced in \cite{MeNa2005} (see \S\ref{s2} below).  
In this note we compute the Levi order 
of minimal 
orbits,  by using the description 
of their  CR  algebras given by 
cross-marked Satake diagrams,  improving and completing %the 
results 
of \cite{Fe2007}.

\par
Other interesting classes of homogeneous CR manifolds were considered by many Authors,  as those of hypersurface type (see  e.g. \cite{DuMeTh2022,MeSp2012, Sa2020, SyZe2022}),  or those  maximally symmetric simple with associated symbol of lenght $2$ 
(see  e.g. \cite{Gr2021}),  or those which are reductive compact 
(see e.g. \cite{AlMeNa2013}).
\par
\vspace{3pt}
The  paper is structured as follows.   In \S\ref{s1} we summarize basic notions 
of  CR  geometry and, for later use,  we provide
Lemma~\ref{l1.3},  which states that 
a CR manifold admitting a complex CR fibration cannot be 
 finitely nondegenerate. 
In \S\ref{s2}
we 
focus on
homogeneous  CR  manifolds; the notions of the previous section 
are rewritten 
 in terms of \textit{ CR  algebras}: these objects are pairs
  $(\gt^{\sigmaup},\qt_{\phiup})$,  
  consiting 
  of a real Lie algebra 
 $\gt^{\sigmaup}$,
 whose elements are infinitesimal  
  CR  automorphism of the manifold,
  and a complex Lie subalgebra 
 $\qt_{\phiup}$ of its complexification
$\gt\,{\doteq}\C{\otimes}\gt^{\sigmaup}$, which encodes its $\Gf^{\sigmaup}$-invariant  CR  structure.  
 In \S\ref{s3} we describe the CR algebras of 
 orbits of real forms in complex flag manifolds,  concentrating upon the unique closed one (cf. e.g.~\cite{AlMe2006}).  
 In \S\ref{s4} we characterize  
 those which are finitely nondegenerates (cf. \cite[Sec.5]{Fe2007}),  
 and, adapting a more general result \cite[Th.2.11]{MaMeNa2021}, 
 we show that   $2$ 
 is an upper bound 
 for the length of the chain of kernels of iterated Levi's forms.
 In \S\ref{s5} we describe 
 canonical CR  fibrations in terms of the subsets of simple roots 
  $\phiup$ employed to
 define the parabolic  CR  algebra $(\gt^{\sigmaup},\qt_{\phiup})$.
In \S\ref{s6}, by using the results of the previous %two 
sections, 
 we give a complete classification of all closed simple homogeneous  CR  manifolds 
 which are $1$-  and $2$-nondegenerate.
 The final classification is summarized in  Table ~\ref{finite nnndg} and Table ~\ref{finite nnndg ex}.

\section{Preliminaries on  CR  manifolds}\label{s1}

A   CR  manifold of type $(n,d)$ is
a pair $(\sfM, \T^{0,1}\sfM)$,  consisting of a 
smooth manifold $\sfM$ of real dimension $2n{+}d,$ and a rank $n$ 
smooth complex  subbundle
$\T^{0,1}\sfM$ of its complexified tangent bundle $\T^{\C}\sfM{\doteq}\C\,{\otimes}\,\T{M},$ satisfying 
\begin{enumerate}
\item  $\T^{0,1}\sfM\cap\overline{\T^{0,1}\sfM}=\{0\}$;
\item  $[\Gamma^{\infty}(\sfM,\T^{0,1}\sfM),\Gamma^{\infty}(\sfM,\T^{0,1}\sfM)]\subseteq
 \Gamma^{\infty}(\sfM,\T^{0,1}\sfM).$
\end{enumerate}
The integers $n,d$ are its \emph{CR-dimension} and  \emph{CR-codimension}, respectively, and
$(2)$ is the  \emph{ formal integrability} condition.
\begin{ntz}
We use the following notation:
\begin{itemize}
\item $\T^{1,0}\sfM\doteq\overline{\T^{0,1}\sfM}$;
\item $\Ha^{\C}\sfM\doteq \T^{1,0}\sfM\oplus \T^{0,1}\sfM$;
\item $\Ha\sfM\doteq \Ha^{\C}\sfM\cap \T\sfM.$
\end{itemize}
\end{ntz}
The rank $2n$ real subbundle $\Ha\sfM$ of $\T\sfM$ is the \emph{real contact distribution}
underlying the  CR  structure of $\sfM.$ Condition (1) above 
defines a 
smooth complex structure $\Jd_{\sfM}$  on the fibers of $\Ha\sfM$ by
\begin{equation}\label{CRstructure}
\T^{0,1}\sfM{=}\{X+i\,\Jd_{\sfM} {X}\mid X
\in\Ha\sfM\}.
\end{equation}
The map $\Jd_{\sfM}$ squares to $-\Id_{\Ha\sfM}$ and is called the \emph{partial complex structure}
of~$\sfM.$ 
An equivalent definition of the  CR  structure can be given by assigning first an even dimensional
distribution $\Hq$ and then a smooth partial complex structure $\Jd_{\sfM}$ on $\Hq$ 
in such a way that
the complex distribution \eqref{CRstructure} satisfies the integrability condition  $(2)$. \par 

\begin{dfn} 

A   CR-manifold $(\sfM, \Ha\sfM, \Jd_{\sfM})$ is \emph{(locally) embeddable} if admits a  (local)  
 CR-embedding into a complex manifold  $(\hat{\sfM},\Jd)$,   
 i.e.    $\sfM\,{\hookrightarrow}\,\hat{\sfM}$ is a smooth (local) embedding and 
 $\Ha_x\sfM  = \T_x\sfM \cap\Jd \T_x\sfM$,  $\Jd_{\sfM}=\Jd|_{\Ha\sfM}$.
 
\end{dfn}
Every real-analytic   CR-manifold $(\sfM, \Ha\sfM, \Jd_{\sfM})$ admits a   
CR-embedding into a complex manifold  $(\hat{\sfM},\Jd)$  (see 
\cite{AnFr1979}).
\begin{ntz}
We denote by $\Hq$ (resp. $\Tq,$ $\Tq^{\C},$ $\Hq^{\C}$, $\Tq^{0,1},$ $\Tq^{1,0}$)
the sheaf of germs of smooth sections of $\Ha\sfM$ 
(resp. $\Ta\sfM,$ $\Ta^{\C}\sfM,$ $\Ha^{\C}\sfM,$ $\Ta^{0,1}\sfM,$ $\Ta^{1,0}\sfM$).
\end{ntz} 
\begin{dfn}\label{fundamental}
 A  CR  manifold $\sfM$ is called \emph{fundamental} at its point $x$ 
  if $\Hq_{x}$
generates under %the 
Lie brackets the Lie algebra $\Tq_{x}.$ 
\end{dfn}
The \emph{Levi form} of $(\sfM,\T^{0,1}\sfM)$ 
at  the point $x\in\sfM$ is the Hermitian symmetric map
\begin{equation}\label{e1.2}
{\mathcal{L}_x}:\T_{x}^{0,1}\sfM\otimes \T_{x}^{1,0}\sfM\rightarrow 
%\T_x\sfM^{\C}\diagup \Ha_{x}^{\C}\sfM\,
\T_x^{\C}\sfM/ \Ha_{x}^{\C}\sfM\,
\end{equation}
defined by 
\begin{equation}\label{e1.3}
{\mathcal{L}_x}(Z,\overline{W})= \frac{1}{2i}\,
\piup ([\utilde{Z},\overline{\utilde{W}}]_{x} ),\;\;\;\forall Z,W\in\T_{x}^{0,1}\sfM,
\end{equation}
where $\piup$ is the canonical projection 
$\T_{x}^{\C}\sfM\rightarrow T_{x}^{\C}\sfM/ \Ha^{\C}_{x}\sfM,$ and  
$\utilde{Z},\utilde{W}\in \Tq_x^{0,1}$ are smooth germs
with $\utilde{Z}_{x}\,{=}\,Z$,  $\utilde{W}_{x}\,{=}\,W$.
\par 
The \emph{Levi kernel} at $x\in M $ is the null space of 
the Levi form
\begin{equation*}
%\Fq_p
\F_{x}\sfM
=\{Z\in \T_{x}^{0,1}\sfM
\,|\,\mathcal{L}_x(Z,\overline{W})=0, \ \forall \, {W}\in \T_x^{0,1}\sfM \}.
\end{equation*}
\par
When the Levi kernel is not trivial, one needs to consider \emph{higher order Levi forms}.
A formulation analogous to \eqref{e1.2}, \eqref{e1.3} would involve jet bundles.
We may simplify the argument by considering 
germs of smooth sections. \par
For each 
point $x\in\sfM$  and any germ $\utilde{Z}\,{\in}\Tq^{0,1}_{x}$ with $\utilde{Z}_{x}{\neq}0$,
we can take the infimum $\ktt(x,\utilde{Z})$ of the set of positive integers $k$ for which 
there are germs $\utilde{Z}_{1},\hdots,\utilde{Z}_{k}\,{\in}\,\Tq^{0,1}_{x}$ such that\footnote{$\ktt(x,\utilde{Z})={+}\infty$
if there is not such an integer $k$.}
\begin{equation} \label{e1.4}
 [\utilde{Z}_{1},[\utilde{Z}_{2},\hdots,[\utilde{Z}_{k},\bar{\utilde{Z}}]]]_{x}\notin \Ha^{\C}_{x}\sfM.
 \end{equation}
 \par
 The \emph{nondegeneracy} of the higher order Levi form is measured by 
\begin{equation}
 \ktt_{x}=\sup\{\ktt(x,\utilde{Z})\mid \utilde{Z}\in\Tq^{0,1}_{x},\; \utilde{Z}_{p}{\neq}0\}.
\end{equation}
\par 
To investigate the degeneracy/nondegeneracy of the Levi form, 
Freeman in \cite[Thm.3.1]{Fr1977} defined 
a nested sequence of sheaves of germs of
smooth complex valued   vector fields on $\sfM$ 

\begin{equation}\label{Freeman sequence}
 \Fq^{(0)}\supseteq\Fq^{(1)}\supseteq\cdots\supseteq\Fq^{(k)}
 \supseteq\Fq^{(k+1)}\supseteq\cdots
\end{equation}
by setting 
\begin{equation} 
\begin{cases}
 \Fq^{(0)}{=}\Tq^{0,1},\\[3pt]
 \Fq^{(k)}{=}{\bigsqcup}_{{x\in\sfM}}\left\{\utilde{Z}\in{\Fq_{x}^{(k-1)}}\left| \; [\utilde{Z},\Tq_{x}^{1,0}]
 {\subseteq}\Fq_{x}^{(k-1)}\oplus\Tq^{1,0}_{x}\right\}\right., 
 \;\text{for $k{\geq}1.$}
\end{cases}
\end{equation}\par
The $k$-th order  Levi form  may be defined by the map 
\begin{equation}
{\mathcal{L}^{(k)}_{x}}:\Fq^{(k-1)}_{p}\otimes \Tq_{p}^{0,1}\rightarrow 
\left(\Fq^{(k-2)}_{x}\oplus \Tq_{x}^{1,0}\right)%\diagup
/\left(\Fq^{(k-1)}_{x}\oplus \Tq_{x}^{1,0}\right),
\end{equation}
with
\begin{equation}
{\mathcal{L}^{(k)}_{x}}(\utilde{Z},\overline{\utilde{W}})= \frac{1}{2i}\,
\piup^{(k)} ([\utilde{Z},\overline{\utilde{W}}] ),
\end{equation}
where $\piup^{(k)}$ is the canonical projection 
\begin{equation}
\piup^{(k)}{:}\Fq^{(k-2)}_{x}\oplus\Tq_{x}^{1,0}\to (\Fq^{(k-2)}_{x}\oplus\Tq_{x}^{1,0})\diagup(\Fq^{(k-1)}_{x}\oplus \Tq_{x}^{1,0}),
\end{equation}
 and  $\utilde{Z},\utilde{W}$ are smooth germs in  $\Fq_{x}^{(k-1)}$ and $\Tq_{x}^{0,1}$, respectively.

\begin{dfn}\label{def3.2}
A  CR  manifold $\sfM$ is \emph{$\ktt$-nondegenerate} at $p$ iff there is a positive integer 
$\ktt$ such that
\begin{equation}\label{k-nndg}
\Fq_{p}^{(\ktt-1)}{\supsetneqq}\Fq_{p}^{(\ktt)}{=}\{0\}.
\end{equation}
In this case $\sfM$ is called \emph{finitely nondegenerate}, otherwise \emph{holomorphically degenerate}.
The integer $\ktt$ of \eqref{k-nndg} is called its \emph{Levi order}. 
\end{dfn}
The standard notion of Levi \emph{nondegeneracy} 
in literature is equivalent to $1$-nondegeneracy, or Levi order $1$, in definition \eqref{def3.2}.  

\begin{exam}[\cite{MeSp2012}]\label{exam_intro}
The \emph{tube over the future light cone}
\begin{equation*}
\sfM\doteq\{(z_1,z_2, z_{3})\in\C^3\,|\,  (\re{z}_1)^2+(\re{z}_2)^2-(\re{z}_3)^2=0\,,
\re{z}_3>0\}\subset\C^3
\end{equation*}
is a hypersurface in $\C^{3}$, and hence 
has a  CR  structure of type $(2,1)$,
being a smooth hypersurface immersed in $\C^{3}$,
which  is 
$2$-nondegenerate.
\end{exam}\par\smallskip
If $\sfM$ and $\sfM'$ are CR manifolds,  a  CR map $f\,{:}\,\sfM\to\sfM'$ is a $\Ci$-smooth
map such that $f_{*}^{\C}(T_{x}^{1,0}M)\,{\subseteq}\,T_{f(x)}^{1,0}\sfM'$ for all $x\,{\in}\,\sfM$. \par
A smooth fibration $\piup\,{:}\,\sfM\to\sfM'$ is said to be CR if 
$\piup_{*}^{\C}(T_{x}^{1,0}M)\,{=}\,T_{f(x)}^{1,0}\sfM'$ for all $x\,{\in}\,\sfM$.\par
We recall the following: 
\begin{lem}[Prop.4.1\cite{AlMeNa2006}]\label{l1.3} 
 Let $\sfM$ and $\sfM^{\,\prime}$ be  CR  manifolds. Assume that 
$\sfM^{\,\prime}$
is locally
embeddable and that there exists a  CR  fibration 
 \begin{equation*}
\piup:\sfM\to\sfM^{\,\prime}
\end{equation*}
with 
 complex
fibers of positive dimension. Then $\sfM$ is holomorphically degenerate. \qed
\end{lem}

\section{Homogeneous CR manifolds and CR algebras}\label{s2}
Let $\Gf^{\sigmaup}$ be a real Lie group of  CR  diffeomorphisms acting transitively on a  CR  manifold $\sfM.$
Fix a point $x$ in $\sfM$ and let $$\piup\,{:}\,\Gf^{\sigmaup}\,{\ni}\,g\,{\to}\,g{\cdot}x\,{\in}\,\sfM$$
be the natural projection. Its  differential at $p$ maps the Lie algebra $\gt^{\sigmaup}$
of $\Gf^{\sigmaup}$ onto the tangent space to $\sfM$ at $x.$ 
By the \emph{formal integrability} condition,  
the pullback of $\T_x^{0,1}\sfM$ by the complexification $\hat{\pi}_*:\gt\rightarrow T^{\C}_x\sfM$ of 
$\pi_*:\gt^{\sigmaup}\rightarrow T_x\sfM$ is complex a Lie subalgebra $\qt$ of $\gt\doteq\gt^{\sigmaup}\otimes\C$:
\begin{equation}
\qt\doteq\hat\pi^{*}(\T_x^{0,1}\sfM).
\end{equation}
We will denote by $\sigmaup$ the anti-$\C$-linear involution of $\gt$ fixing $\gs$.
\par\smallskip
A different choice of the base point $x$ would yield another  $\Ad(\Gf^{\sigmaup})$-con\-ju\-gated
complex Lie subalgebra.  
\par\smallskip
  Vice versa,  the assignment  of a complex
Lie subalgebra $\qt$ of $\gt$  
yields a $\Gf^{\sigmaup}$-equivariant 
structure of CR-manifold on a (locally) 
$\Gf^{\sigmaup}$-homogeneous space
 $\sfM$, by requiring that $\Tq_{x}^{0,1}$ is generated by 
 the pushforward of $\qt$ (see \cite{MeNa2005} for more details).

\begin{dfn}
A \emph{  CR  algebra} is a pair  $(\gt^{\sigmaup},\qt)$
consisting of a real Lie algebra $\gt^{\sigmaup}$ and  
a complex Lie subalgebra $\qt$ of its complexification
$\gt{=}\C{\otimes}_{\R}\gt^{\sigmaup},$ 
 such that the quotient $\gt^{\sigmaup}/(\gt^{\sigmaup}\cap\qt)$ is a finite dimensional real vector space. \par
 \end{dfn}

The
 real Lie 
 algebra $\gt^{\sigmaup}$ 
 encodes 
 the transitive group of  CR  diffeomorphisms 
 and 
 the 
 complex Lie 
 subalgebra
 $\qt$  
 the  CR  structure.
The CR-di\-men\-sion $n$ and 
CR-codimension $d$ of $\sfM$ are expressed 
in terms of its associated  CR  algebra 
$(\gt^{\sigmaup},\qt)$ by 
\begin{equation} \label{1.1}
\begin{cases} n=\dim_{\C}\qt-\dim_{\C}(\qt{\cap}\sigmaup(\qt)),\\
d=\dim_{\C}\gt - \dim_{\C}(\qt+\sigmaup(\qt)).
\end{cases}
\end{equation} 

\par
A
CR  algebra $(\gt^{\sigmaup},\qt)$ is called
\begin{itemize}
\item\emph{totally real}: if  $\qt=\qt\cap\sigmaup(\qt)$,  i.e.
the  CR  dimension $n$ is $0$;
\item \emph{totally complex}: if $\qt+\sigmaup(\qt)=\gt$,  i.e.  the  CR 
codimension $d$ is $0$.
\end{itemize}\par 
 We call the intersection $\qt\cap\gt^{\sigmaup}$  its
 \emph{isotropy subalgebra},   and the subspace
$\Hq=(\qt+\sigmaup(\qt))\cap\gt^{\sigmaup}$ its \emph{holomorphic tangent space}.  
Moreover,  a  CR  algebra $(\gt^{\sigmaup},\qt)$ is said to be \emph{effective}  
if  no nontrivial ideal of $\gt^{\sigmaup}$ is contained in $\gt^{\sigmaup}\cap\qt$.  \par 
The strong non integrability condition for the CR distribution of Definition\,\ref{fundamental}
translates to the following notion for CR algebras.
\par \begin{dfn} A CR  algebra $(\gt^{\sigmaup},\qt)$ 
is called \emph{fundamental} 
if 
$\qt\,{+}\,\sigmaup(\qt)$ generates $\gt$ as a Lie algebra. 
\end{dfn}

For a  finitely nondegenerate 
(locally) homogeneous   CR  manifold $\sfM$ the Levi order $\ktt$ can be computed
by using its associated  CR  algebra $(\gt^{\sigmaup},\qt)$.  
Noticing that $\gt\,{\cap}\,\sigmaup(\qt)$ is the 
 complexification of its isotropy subalgebra $\gt^{\sigmaup}{\cap}\,\qt$, 
 the $1$-nondegeneracy of the Levi form  can be
 stated by 
\begin{equation}
 \forall{Z}\in\qt{\setminus}\sigmaup(\qt), \;\;\exists\, Z'\in{\qt} \;\;\;\text{such that}\;\;\; [Z,\sigmaup({Z}')]\notin\qt
 +\sigmaup(\qt),
\end{equation}
and this is equivalent to 
\begin{equation*}
 \qt^{{(1)}}\doteq\{Z\in\qt\mid [Z,\sigmaup(\qt)]\subseteq\qt+\sigmaup(\qt)\}=\qt\cap\sigmaup(\qt).
\end{equation*}

\par
Similarly  to \eqref{e1.4},
in the homogeneous case one can seek whether, for any given  
$Z\,{\in}\,\qt{\backslash}(\qt{\cap}\sigmaup(\qt))$ one can find a positive integer $k$ and 
${Z}_{1},\hdots,{Z}_{k}\,{\in}\,\qt$ such that
 $$[\overline{Z}_{1},\hdots,\overline{Z}_{k},Z]\,{\notin}\,
\qt{+}\sigmaup(\qt).$$To this aim, 
in analogy to 
the
Freeeman's sequence \eqref{Freeman sequence}
of \S\ref{s1},  
it is  convenient to consider 
the descending chain (see e.g.  \cite{ Fe2007,  Fr1977, KaZa2006}) 
\begin{equation}\label{e1.1}
\qt^{(0)}\supseteq\qt^{(1)}\supseteq\cdots\supseteq\qt^{(k-1)}\supseteq\qt^{(k)}
\supseteq\qt ^{{(k+1)}}\supseteq \cdots,
\end{equation}
with 
\begin{equation*}
 \begin{cases}
\qt^{(0)}=\qt,\\
\qt^{(k)}=\{Z\in\qt^{(k-1)}\mid [Z,\sigmaup(\qt)]\subseteq\qt^{(k-1)}+\sigmaup(\qt)\},\;\;\text{for $k{\geq}1.$}
\end{cases}
\end{equation*} \par
Note that $\qt{\cap}\sigmaup(\qt)\,{\subseteq}\,\qt^{(k)}$ for all integers 
$k{\geq}0.$
If $(\gt^{\sigmaup},\qt)$ is fundamental and $\qt/(\qt{\cap}\sigmaup(\qt))$ is finite dimensional,
there is a smallest nonnegative 
integer $\ktt$ such that $\qt^{(k')}\,{=}\,\qt^{(\ktt)}$ for all $k'{\geq}\ktt.$ 

\begin{dfn}\label{knndg}
A   CR  algebra $(\gt^{\sigmaup},\qt)$ is said to be \emph{$\ktt$-nondegenerate} iff 
\begin{equation}
\qt^{(\ktt-1)}{\supsetneq}\,\qt^{(\ktt)}
 {=}\,\qt{\cap}\sigmaup({\qt}).
 \end{equation}
\end{dfn}
\begin{prop}[\cite{MaMeNa2021}]\label{Prop4.1}
 The terms $\qt^{(k)}$ of  \eqref{e1.1} are Lie subalgebras of~$\qt.$\qed
\end{prop} 

\begin{rmk}[\cite{MaMeNa2021},\cite{MaMeNaSp2017}]
We point out that the notion of 
\emph{finite nondegeneracy} 
defined in 
\cite{MeNa2005} by the requirement that for any complex Lie subalgebra $\qt '$ of~$\gt,$
\begin{equation} \label{e2.6}
 \qt\subseteq\qt '\subseteq\qt +\sigmaup(\qt) \;\Longrightarrow \qt '=\qt,
\end{equation}
is equivalent to that of Definiton\,\eqref{knndg}.
\par
Indeed, it easily follows from \cite[Lemma 6.1]{MeNa2005} that 
\begin{equation*}
 \qt'=\qt+\sigmaup({\qt}^{(\infty)}),\;\;\;\text{with}\;\;\; \qt^{(\infty)}={\bigcap}_{k{\geq}0}\qt^{(k)}
\end{equation*}
 is the largest complex Lie subalgebra $\qt '$ of $\gt$ with
 $\qt\,{\subseteq}\,\qt '\,{\subseteq}\,\qt+\sigmaup({\qt}).$  
 \end{rmk}
\section{Structure of the closed orbit}\label{s3}

A real Lie algebra $\gt^{\sigmaup}$ is  a \emph{real form} 
of a complex Lie algebra $\gt$ if  
\begin{equation*}
 \gt\simeq\C{\otimes}_{\R}\gs.
\end{equation*}\par
The 
real forms of $\gt$ are the 
fixed points loci  \begin{equation}
 \gt^{\sigmaup}=\{X\in\gt \,|\,  \sigmaup (X)\!=\!X\}
 \end{equation}
of its anti-$\C$-linear involutions $\sigmaup$.\par\smallskip
A complex flag manifold $\sfF$ is
a smooth compact algebraic variety that can be described as the quotient
of a
complex semisimple Lie group $\Gf$ by  a parabolic subgroup $\Qf$.  
In \cite[Theorem 2.6]{Wo1969}  J.  A.  Wolf shows the action of 
any real form $\Gf^{\sigmaup}$ of $\Gf$ partitions $\sfF$ into finitely many orbits.
With the partial complex structures induced by $\sfF,$
these orbits make a nice class of homogeneous
 CR  manifolds that were studied by many authors 
(see e.g. 
\cite{AlMe2006, Fe2007, MaMeNa2021,MaNa2017}).
Among
these orbits, only one is 
closed; it is also compact, connected    
and of minimal dimension \cite[Theorem 3.3]{Wo1969}.\par
Being connected and simply connected, a
complex flag manifold $\sfF{=}\Gf/\Qf$ is 
completely described by the Lie pair
$(\gt,\qt)$ consisting of the complex Lie algebras of $\Gf$ and $\Qf$
and vice versa to any  Lie pair $(\gt,\qt)$ of a complex semisimple Lie algebra and
a parabolic subalgebra $\qt$ 
corresponds a unique flag manifold $\sfF.$ 
Therefore  the classification of complex flag manifolds reduces to that
of parabolic subalgebras of semisimple complex Lie algebras. \par
Parabolic subalgebras  
$\qt$  of $\gt$ are classified, modulo automorphisms,  by a finite set of parameters as follows.  
Having fixed a Cartan subalgebra $\hg$ of~$\qt$, 
a Weyl chamber $C$ of the root system $\Rad$ of $(\gt,\hg)$,
the classes of parabolic subalgebras of $\gt$ are in a one to one correspondence  
with the subsets $\phiup$
of the basis $\Bz(C)$ 
of simple positive  roots for the lexicographic  order of 
 $C$ 
(see e.g. \cite[Ch.VIII,\S{3.4}]{Bo2005}). In the following we will drop, for simplicity,
the explicit reference to $C$, writing e.g. $\Bz$ instead of $\Bz(C)$.
\par
Each root $\beta$ in $\Rad$ can be written in a unique way as a 
notrivial linear combination 
\begin{equation}
 \beta={\sum}_{\alpha\in\mathcal{B}}n_{\beta,\alpha}\alpha,
\end{equation}
with integral coefficients $n_{\beta,\alpha}$ which are either all $\geq{0},$
or all $\leq{0}$.  Set 
\begin{equation}
\supp(\beta)=\{\alpha\in\mathcal{B}\mid n_{\beta,\alpha}\neq{0}\}.
\end{equation} 
The parabolic subset and the parabolic subalgebra associated to $\phiup\,{\subseteq}\,\Bz$ are 
\begin{gather}
\label{e3.4}
 \Qc_{\;\phiup}{=}\{\beta\in\Rad\mid n_{\beta,\alpha}\geq{0},\,\forall\alpha\in\phiup\}\subseteq\Rad,\\
 \label{e3.5}
 \qt_{\phiup}{=}\hg\oplus{\sum}_{\beta\in\Qc_{\;\phiup}}\gt_{\beta},\;\;\text{with}\;\;
 \gt_{\beta}\,{=}\,\{Z\,{\in}\,\gt\,{\mid}\,[H,Z]\,{=}\,\beta(H)Z\}. 
\end{gather}

The fact that 
$\Qc_{\;\phiup}$ is \emph{parabolic} 
means 
that
$$(\Qc_{\;\phiup}{+}\Qc_{\;\phiup})\cap\Rad\subseteq\Qc_{\;\phiup}\;\;\text{and}\;\;   
%(\Qc_{\;\phiup}+\Qc_{\;\phiup})\cap\Rad\subseteq\Qc_{\;\phiup}\;\; and \;\; 
\Qc_{\;\phiup}\cup(-\Qc_{\;\phiup})=\Rad.$$

\begin{nnott}\label{n1}
We denote by 
$\xiup_{\phiup}$ be the $\R$-linear functional on the $\R$-linear span $\sfE$ 
of $\Rad$ which satisfies 
\begin{equation}\label{xiup}
\xiup_{\phiup}(\alpha)=
\begin{cases}
1, & \text{if } \alpha\in\phiup, \\
0, & \text{if } \alpha\in\mathcal{B}{\backslash}\phiup.
\end{cases}
\end{equation}
Then 
\begin{equation}
 \Qc_{\;\phiup}=\{\beta\in\Rad\mid \xiup_{\phiup}(\beta)\geq{0}\}
\end{equation}
and we get the partitions
\begin{equation} \label{equ2.5}
\Qc_{\;\phiup}=\Qc^{r}_{\,\,\phiup}\cup\Qc^{n}_{\,\,\phiup},\;\;\;
\Rad =\Qc^{r}_{\,\,\phiup}\cup\Qc^{n}_{\,\,\phiup}\cup\Qc^{-n}_{\,\,\phiup},\;\;
\end{equation}
where,
\begin{itemize}
 \item $\Qc^{r}_{\,\,\phiup}\doteq\{\beta\in\Qc_{\;\phiup}\mid -\beta\in\Qc_{\;\phiup}\}=
 \{\beta\in\Rad\mid \xiup_{\phiup}(\beta)=0\},$
 \item $\Qc^{n}_{\,\,\phiup}\doteq\{\beta\in\Qc_{\;\phiup}\mid -\beta\notin\Qc_{\;\phiup}\}=
 \{\beta\in\Rad\mid \xiup_{\phiup}(\beta)>0\} ,$
  \item $\Qc^{-n}_{\,\,\phiup}\doteq\{\beta\in\Rad\mid -\beta\in\Qc^{n}_{\,\,\phiup}\}=
 \{\beta\in\Rad\mid \xiup_{\phiup}(\beta)<0\}.$
\end{itemize}

We recall (see e.g. \cite[Ch.VIII,\S{3}]{Bo2005}):
\begin{itemize}
 \item $\qt^{r}_{\phiup}=\hg\oplus{\sum}_{\alpha\in\Qc^{r}_{\phiup}}\gt_{\alpha}$ is a reductive
 complex Lie algebra;
 \item $\qt^{n}_{\phiup}={\sum}_{\alpha\in\Qc^{n}_{\;\phiup}}\gt_{\alpha}$ is the nilradical of $\qt_{\phiup}$; 
 \item $\qt_{\phiup}=\qt^{r}_{\phiup}\oplus\qt^{n}_{\phiup}$ is a Levi-Chevalley decomposition of $\qt_{\phiup}$;
% (see \cite[Sec.2.5]{MaMeNa2020});
 \item $\qt_{\phiup}^{-n}={\sum}_{\alpha\in\Qc^{-n}_{\,\,\phiup}}\gt_{\alpha}$
 is a nilpotent Lie algebra;
 \item 
 $\qt^{\text{opp}}_{\phiup}=\qt^{r}_{\phiup}\oplus\qt^{-n}_{\phiup}$ is the
 parabolic Lie subalgebra of $\gt$ \emph{opposite of $\qt_{\phiup}$},
 decomposed into the direct sum of its reductive subalgebra $\qt^{r}_{\phiup}$ and
its nilradical $\qt^{-n}_{\phiup}.$
\end{itemize}
\end{nnott}

\smallskip\par
Given a Cartan subalgebra $\hs$ of a semisimple real Lie algebra
$\gs$, we can find a \textit{compatible} Cartan decomposition
$\gt^{\sigmaup}=\kt\oplus\pt$ of $\gs$, such that $\hs\,{=}\,(\hs\,{\cap}\,\kt)\,{\oplus}\,(\hs\,{\cap}\,\pt)$.
The two summands are its \textit{toroidal} and  \textit{vector} parts.
Among the equivalence classes of Cartan subalgebras of $\gs$, there are only one with maximal toroidal part
and only one with maximal vector part. \par

 A real Lie subalgebra $\mathfrak{t}^{\sigmaup}$ of $\gt^{\sigmaup}$ is \emph{triangular} if all linear maps $ad_{\gt^{\sigmaup}}(X)\in \mathfrak{gl}_{\R}(\gt^{\sigmaup})$
with $X\in\mathfrak{t}^{\sigmaup}$ can be simultaneously represented by triangular matrices in a suitable
basis of $\gt^{\sigmaup}$.  
All maximal triangular subalgebras of $\gt^{\sigmaup}$ are conjugate by an inner
automorphism (\cite{Mo1961}, \S{5.4}).  
A real Lie subalgebra of $\gt^{\sigmaup}$ containing a maximal
triangular subalgebra of $\gt^{\sigmaup}$ is called a $\mathfrak{t}^{\sigmaup}$\emph{-subalgebra}.
An effective parabolic CR algebra $(\gt^{\sigmaup}, \qt_{\phiup})$  
will be called \emph{minimal} if  $\qt\cap\gt^{\sigmaup}$ is
a $\mathfrak{t}$-subalgebra of $\gt$.
\begin{thm}[{\cite[Th. 5.8]{AlMeNa2006}}]
The  CR  manifold  associated to an effective parabolic
subalgebra $(\gt^{\sigmaup}, \qt_{\phiup})$ is closed if and only if $\gt^{\sigmaup}\cap\qt_{\phiup}$ is a $\mathfrak{t}^{\sigmaup}$-subalgebra of $\gt^{\sigmaup}$, i.e. $(\gt^{\sigmaup}, \qt_{\phiup})$ is minimal.
\end{thm}
\begin{proof}
Indeed, since $\Gf^{\sigmaup}$ is a linear group, a $\Gf^{\sigmaup}$-homogeneous space is closed 
if and only if the isotropy $\Gf^{\sigmaup}\,{\cap}\,\Qf_{\,\phiup}$ 
contains a maximal connected triangular subgroup (see \cite[
Part\,{II}, Ch.5, \S{1.1}]{On1993}), i.e. 
if $\gt^{\sigmaup}\cap\qt_{\phiup}$ is a $\mathfrak{t}^{\sigmaup}$-subalgebra of $\gt^{\sigmaup}.$
\end{proof}
\begin{nnott}
Fix a 
%maximally vectorial 
Cartan subalgebra $\mathfrak{h}^{\sigmaup}$ of $\gs$. Its complexification $\mathfrak{h}$
is a Cartan subalgebra of $\gt$. The conjugation $\sigmaup$ induces a symmetry
on $\Rad\,{=}\,\Rad(\gt,\hg)$, that we will still denote by the same symbol $\sigmaup$.

A root $\alpha\in\Rad$ is called :
\begin{itemize}
\item \emph{real} if $\sigmaup (\alpha)=\alpha$;
\item \emph{immaginary} if $\sigmaup (\alpha)=-\alpha$;
\item  \emph{complex} if $\sigmaup (\alpha)\neq
\pm\alpha$. 
\end{itemize}
We  denote by $\Rad_{\;\bullet}^{\sigmaup}$ the set of 
imaginary roots in $\Rad$.  
\end{nnott}

We say that $C$ is an $S$-chamber if the 
$\sigmaup$-conjugate of any positive complex 
root stays  positive: 
\begin{equation}\label{Adapted Weyl}
 \sigmaup(\alpha)\in\Rad^{+},\;\;\forall \alpha\in\Rad^{+}{\setminus}\Rad_{\;\bullet}^{\sigmaup}.
\end{equation}
We can always find Weyl chabers with this property
(see e.g. \cite[Proposition\,{6.1}]{AlMeNa2006})
and, on the basis $\Bz$ of simple positive roots of an $S$\!-chamber  
an involution $\epi\,{:}\,\mathcal{B}{\to}\mathcal{B}$ is defined, 
which keeps fixed
the elements of 
$\mathcal{B}^{\sigmaup}_{\bullet}\,{\doteq}\,\mathcal{B}\,{\cap}\,\Rad_{\;\bullet}^{\sigmaup}$ 
and    
such that, for nonnegative $n_{\alpha,\beta}\,{\in}\,\Z,$  
\begin{equation} \label{sigmaup}
\begin{cases}
\sigmaup({\alpha})=-\alpha, & \forall\alpha\in\Bz^{\sigmaup}_{\bullet},\\
\sigmaup({\alpha})=\epi(\alpha)+{\sum}_{\beta\in\Bz^{\sigmaup}_{\bullet}}
n_{\alpha,\beta}\beta,
&\forall \alpha\in\mathcal{B}{\backslash}\Bz^{\sigmaup}_{\bullet}.
\end{cases}
\end{equation}
\par 
\begin{dfn}
\smallskip
A  \emph{Satake  diagram} $\mathcal{S}$ is obtained from  a Dynkin 
diagram, with  the  simple roots of an $S$-chamber, 
 by painting black the roots in
$\Bz^{\sigmaup}_{\bullet}$ and joining by an 
 arch the pairs of distinct complex roots $\alpha,\beta$ 
with $\epi(\alpha)=\beta.$ 
\end{dfn} 
The Satake diagrams associated to an $\hs$ with maximal vector part 
classify (modulo equivalence) the
real forms of the complex Lie algebra 
dexcribed by the underlying 
 the Dynkin diagram  (see e.g.  \cite{Ar1962}).  
 \par 
 Fix a subset $\phiup$ of $\mathcal{B}$ and consider the diagram $(\mathcal{S}, \phiup)$ obtained from $\mathcal{S}$ by adding a \emph{cross-mark} on each node of S corresponding to a root in $\phiup$.

\begin{exam} \label{ex2.5}
Fix the real form $\gt^{\sigmaup}=\mathfrak{su}(1,3)$, then  the  CR  algebra $(\gt^{\sigmaup},\qt_{\phiup})$,  where $\phiup{=}\{\alpha_2\}{\subset}\{\alpha_1,\alpha_2,\alpha_3\}=\mathcal{B}$,  is described by the \emph{cross-marked Satake diagram}
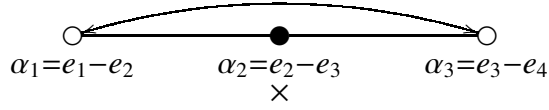
\vspace{1cm}
\begin{equation*} 
	 \xymatrix@R=-.3pc{&\!\!\medcirc\!\!\ar@{-}[r]\ar@{<->}@/^1pc/[rr]
&\!\!\medbullet\!\! \ar@{-}[r]&\!\!\medcirc\!\!\\
&\alpha_1{=}e_1{-}e_2&\alpha_2{=}e_2{-}e_3&\alpha_3{=}e_3{-}e_4
\\
  &&\times&}
\end{equation*}
where the cross-mark indicates the root in $\phiup$.
\par
This is the CR algebra of the minimal orbit $\sfM$ 
of $\SU(1,3)$, consisting of the totally isotropic 
two-planes of $\C^{4}$ for 
a hermitian symmetric form of signature~$(1,3).$ 

Here $\gt\,{\simeq}\,\slt_{4}(\C),$ $\Rad\,{=}\,\{{\pm}(e_{i}{-}e_{j})\,{\mid}\, 1{\leq}i{<}j{\leq}4\},$  with a basis of simple roots 
$$\mathcal{B}\,{=}\,\{e_{1}{-}e_{2},\,e_{2}{-}e_{3},\,e_{3}{-}e_{4}\}$$ for an orthonormal basis
$e_{1},e_{2},e_{3},e_{4}$ of $\R^{4}$.  
Then the linear functional  $\xiup_{\phiup}$ on $\Rad$ in \eqref{xiup} is given by
\begin{equation*}
 \xiup_{\phiup}(e_{i})= 
\begin{cases}
 1, & i{=}1,2,\\
 0, & i{=}3,4,
\end{cases}
\end{equation*}
and the conjugation $\sigmaup$ induced by $\mathfrak{su}(1,3)$ is 
\begin{equation*}
 \sigmaup(e_{1})={-}e_{4},\;\; \sigmaup(e_{i})={-}e_{i}, \,\,\,i{=}2,3.\\
\end{equation*}
\end{exam}\par
\smallskip
Summarizing,  to a parabolic subalgebra $\qt_{\phiup}$ we associate 
a cross-marked Dynkin diagram, consisting of the Dynkin diagram of $\gt$ 
with marks on the nodes corresponding to
roots in $\phiup$. The correspondence 
between isomorphism classes of parabolic subalgebras and cross-marked Dynkin diagrams is bijective.
\par
An effective parabolic minimal  CR  algebra $(\gt^{\sigmaup}, \qt)$ 
can always be described by a cross-marked Satake diagram:  
in fact, we can find a Cartan subalgebra $\mathfrak{h}^{\sigmaup}$ of $\gt^{\sigmaup}$  
with maximal vector part and an $S$-chamber $C$ such that, 
for a subset $\phiup$ of $\Bz$, $\qt\,{=}\,\qt_{\phiup}$.
We have indeed
\begin{thm}[{\cite[Th.6.3]{AlMeNa2006}}]\label{t3.4}
Having fixed a Cartan subalgebra $\hs$ of $\gs$ with maximal vector part, there is
a one to one correspondence 
\begin{equation}
(\mathcal{S},\phiup)\leftrightarrow (\gt^{\sigmaup},\qt_{\phiup})
\end{equation}
between cross-marked Satake diagrams 
(modulo automorphism of cross-marked Satake diagrams) 
and minimal effective parabolic  CR  algebras (modulo  CR  algebra isomorphisms).  
\qed
\end{thm}

\begin{rmk}

Having fixed a $\hs$ with maximal vector part and 
$\phiup\,{\subset}\,\Bz$ for a base of simple positive roots
of a Weyl chamber $C$ satisfying \eqref{Adapted Weyl}
 (see \cite[Prop.6.1,\,6.2]{AlMeNa2006}) on $\Rad$,
 the symmetry  %defines a symmetry 
 $\vartheta(\alpha)\,{=}\,{-}\sigmaup(\alpha)$, 
 corresponding to a Cartan involution on $\gs$, 
leaves invariant
 $\Qc_{\;\phiup}^{-n}\cap\sigmaup\left({\Qc}^{\,n}_{\;\phiup}\right)$:
 this set 
 is 
 therefore the union of 
  its fixed points, which are in $\Rad_{\,\bullet}^{\sigmaup}$,
 and of pairs $(\alpha, \sigmaup({\alpha}))$ of symmetric distinct roots. \par
If   
 $(\gs,\qt_{\phiup})$ is the effective parabolic $CR$ algebra of a minimal orbit, then by \eqref{Adapted Weyl}
 the set of roots $\Qc_{\;\phiup}$ associated to
  $\qt_{\phiup}$ satisfies
\begin{equation}\label{blacknodes}
 \Qc^{n}_{\;\phiup}\cap\sigmaup(\Qc_{\;\phiup}^{-n})\,{\subseteq}\,\Rad_{\,\;\bullet}^{\sigmaup}, 
 \text{or, equivalently }\Qc^{-n}_{\;\phiup}\cap\sigmaup(\Qc_{\;\phiup}^{n})\,{\subseteq}\,\Rad_{\,\;\bullet}^{\sigmaup}.
\end{equation}
\end{rmk}

\begin{rmk}[{\cite[Prop.6.2]{AlMe2006}}]\label{prodotti}
In the following we will further restrict our consideration to the case of a simple $\gs$.  In fact if $(\gt^{\sigmaup},\qt_{\phiup})$ 
is an effective  parabolic 
CR  algebra and $\gt^{\sigmaup}{=}\gt^{\sigmaup}_1\oplus\cdots\oplus\gt^{\sigmaup}_{\ell}$ 
the decomposition of $\gt^{\sigmaup}$ into the direct sum of its simple ideals,  then the following holds:
\begin{enumerate}
\item $\qt=\qt_1\oplus\cdots\oplus\qt_{\ell},$ where $\qt_j=\qt\cap\gt_j$;
\item for each $j\,{=}\,1,\hdots,\ell$, the pair 
$(\gt^{\sigmaup}_j,\qt_j)$ is an effective fundamental parabolic  CR  algebra;
\item $(\gt^{\sigmaup},\qt_{\phiup})$ is fundamental iff $(\gt^{\sigmaup}_j,\qt_j)$ is fundamental for all $j\in\{1,\dots,\ell\}$;
\item $(\gt^{\sigmaup},\qt_{\phiup})$ is finitely nondegenerate iff $(\gt^{\sigmaup}_j,\qt_j)$ is finitely nondegenerate for all $j\in\{1,\dots,\ell\}$.
\end{enumerate}
\end{rmk}

\begin{thm}[{\cite[Th.10.2]{AlMeNa2006}}]\label{t3.6}
A simple effective parabolic minimal CR algebra $(\gt^{\sigmaup}, 
\qt_{\phiup})$ with
associated cross-marked Satake diagram $(\mathcal{S}, \phiup)$ is  
totally complex if and only if 
\begin{equation}\label{e3.13}
 \Qc_{\;\phiup}^{n}\cap\sigmaup(\Qc_{\;\phiup}^{n})\,{=}\,\emptyset,
\end{equation}
i.e. if and only if 
 one of the following holds:

\begin{enumerate}
\item $\gt^{\sigmaup}$ is the \emph{compact real form} of $\gt$;
\item $\gt^{\sigmaup}$ is of the complex type 
(i.e. its complexification $\gt$ is semisimple, but non simple)
and all cross-marked nodes 
belong to the same connected component of $\mathcal{S}$;
\item $(\mathcal{S},\phiup)$ is one of the cases %showed in 
of the following table:
\begin{table}[H]

\begin{tabular}{ | c | c | c | c |} 
  \hline
  Type & $\,\,\,\gs$  &$\mathcal{S}$& $\phiup$\\ 
  
  \hline

    \textsc{A\,II} & $\mathfrak{sl}(\mathbb{H},p),\;\ell{=}2p{-}1$ &
    \dynkin[labels={\alpha_1,\alpha_2,,,\alpha_\ell},scale=2.0]A{II}&
    $\phiup{=}\{\alpha_1\},\{\alpha_\ell\}.$\\ 
  \hline
   \textsc{D\,{II}} & $\mathfrak{so}(1,{2\ell{-}1})$ 
  &
\dynkin[labels={\alpha_1,,,\alpha_{\ell-1},\alpha_{\ell}},scale=1.8]D{II}         
& 
   
  $\phiup{=}\{\alpha_{\ell-1}\};$
   $\phiup{=}\{\alpha_\ell\}.$
  \\ 
  \hline
\end{tabular}\vspace{6pt}
 \caption{\label{complex ones}Totally complex  CR   algebras $(\gt^{\sigmaup},
\qt_{\phiup})$ with simple $\gt$.}
\end{table}
\end{enumerate}
\end{thm}
\begin{proof} Condition \eqref{e3.13} is a consequence of the fact that $\Rad$ is a disjoint union
\begin{equation*}
 \Rad=\left(\Qc_{\;\phiup}\cup\sigmaup(\Qc_{\;\phiup})\right)\cup
 \left(\Qc_{\;\phiup}^{-n}\cap\sigmaup(\Qc_{\;\phiup}^{-n})\right).
\end{equation*}
\par
An open minimal orbit $\sfM$, being also closed, coincides with the flag manifold
$\sfF$. This is always the case 
when $\gs$ is compact (see e.g. \cite{Mo1950,Wo2001}).
Cases (2) and (3)  follow easily from \eqref{e3.13}
(see also \cite[Corollary 1.7]{Wo2001}).
\end{proof}

\section{A bound for the k-nondegeneracy of the closed orbit}\label{s4}
Finite nondegeneracy  of the  
closed orbits  of real forms was characterized in \cite[Th.11.5]{AlMeNa2006} 
in terms of their description by cross-marked Satake diagrams.  
Using Prop.\ref{Prop4.1} we can  discuss the order of finite
nondegeneracy   in  terms of the chain \eqref{e1.1} for $(\gt^{\sigmaup},\qt_{\phiup})$
 and can be investigated  by using 
the combinatorics of the root system. 
Let us set 
\begin{equation}
 \Qc_{\;\phiup}^{(\,k)}=\{\alpha\in\Rad\mid \gt_{\alpha}\subseteq
 \qt_{\,\,\phiup}^{(k)}\}, \;\;\text{so that}\;\;
 \qt_{\,\,\phiup}^{(k)}=\hg\oplus{\sum}_{\alpha\in\Qc_{\;\phiup}^{(k)}}
 \gt_{\alpha}.
\end{equation}
They can be defined recursively by  
\begin{equation} 
\begin{cases}
\Qc_{\;\phiup}^{(0)}=\Qc_{\;\phiup},\\[4pt]
 \Qc_{\;\phiup}^{(1)}=\{\alpha\in\Qc_{\;\phiup}\mid (\alpha+\sigmaup({\Qc}_{\;\phiup}))
 \cap\Rad\subseteq \Qc_{\;\phiup}\cup\sigmaup({\Qc}_{\;\phiup})\},\\[4pt]
 \begin{aligned}
 \Qc_{\;\phiup}^{(k)}=\{\alpha\in\Qc_{\;\phiup}^{(k-1)}\,{\mid}\, 
 (\alpha{+}\sigmaup({\Qc}_{\;\phiup}))\,{\cap}\,\Rad\,{\subseteq}\, \Qc_{\;\phiup}^{(k-1)}\,{\cup}\,
 \sigmaup({\Qc}_{\;\phiup})\},\; 
 \text{for $k{>}1.$}\end{aligned}
\end{cases}
\end{equation}
This yields a characterization of finite  nondegeneracy in terms 
of roots:
\begin{prop}\label{prop5.3}
 A necessary and sufficient condition for $(\gt^{\sigmaup},\qt_{\phiup})$ 
 being finitely nondegenerate of order $k$ 
 is that $\Qc_{\;\phiup}^{\,(k-1)}{\supsetneqq}\,\Qc_{\;\phiup}^{\,(k)}\,{=}\,
 \Qc_{\;\phiup}\,{\cap}\,\sigmaup({\Qc}_{\;\phiup}).$ \qed
\end{prop}
By Proposition \ref{prop5.3} we obtain:
\begin{rmk}[{\cite[Lemma\,12.1]{AlMeNa2006}}]
A necessary and sufficient condition for $(\gt^{\sigmaup},\qt_{\phiup})$ being
 finitely nondegenerate   
is that  for all  $\beta\,{\in}\,\Qc_{\;\phiup}\,{\cap}\,\sigmaup({\Qc}_{\;\phiup}^{-n}),$ 
one can find
$k\in\Z_{+},$ and $\alpha_{1},\hdots,\alpha_{k}\,{\in}\,\sigmaup({\Qc}_{\,\, \phiup})$ 
such that 
\begin{equation}\label{e4.3}
\begin{cases}
 \gamma_h\,{\doteq}\,\beta{+}{\sum}_{i=1}^{h}\alpha_{i}\,{\in}\,\Rad,\;\forall 
  1{\leq} h {\leq} k,\\
 \gamma_k\in\Qc_{\;\phiup}^{-n}\cup\sigmaup({\Qc}_{\;\phiup}^{-n}).
 \end{cases}
\end{equation}
\end{rmk}
\begin{ntz}
Denote by  $\ktt^{\sigmaup}_{\phiup}(\beta)$ the smallest 
positive integer for which \eqref{e4.3} holds true. Set $\ktt^{\sigmaup}_{\phiup}(\beta)\,{=}\,{+}\infty$
if there is no $k$ for which
\eqref{e4.3} is satisfied.
\end{ntz}
\begin{prop}[{\cite[Lemma\,2.6]{MaMeNa2021}}] \label{p4.3} Let
 $\beta\,{\in}\,\Qc_{\;\phiup}\,{\cap}\,\sigmaup({\Qc}_{\;\phiup}^{-n})$ and
 assume that $k\,{=}\,\ktt^{\sigmaup}_{\phiup}(\beta)\,{<}\,{+}\infty$.
 Then any sequence 
 $\alpha_{1},\hdots,\alpha_{k}$ satisfying 
\eqref{e4.3}
has the properties: 
\begin{itemize}
\item[$(i)$]
 $\alpha_{i}\in\sigmaup({\Qc}_{\;\phiup})\,{\cap}\,\Qc_{\;\phiup}^{-n}$ for all $1{\leq}i{\leq}k$;
 \item[$(ii)$] $\beta+{\sum}_{i{\leq}h}\alpha_{i}\in\Qc_{\;\phiup}\,{\cap}\,\sigmaup({\Qc}_{\;\phiup}^{-n})$
 for all $h{<}k$;
 \item[$(iii)$] \eqref{e4.3} is satisfied by all permutations of 
 $\alpha_{1},\hdots,\alpha_{k}$; 
 \item[$(iv)$] $\alpha_{i}{+}\alpha_{j}\,{\notin}\,\Rad$ for all $1{\leq}i{<}j{\leq}k$.\qed
\end{itemize}
\end{prop} 

\begin{exam}
We continue with the Example \ref{ex2.5}.  We obtain 
\begin{align*}
 &\Qc^{-n}_{\,\,\phiup}\cap\sigmaup({\Qc}^{-n}_{\phiup})=\{\alpha_1{+}\alpha_2{+}
 \alpha_3\,{=}\,e_{1}{-}e_{4}\},\\
 &\sigmaup({\Qc}_{\;\phiup})
 \cap\Qc^{-n}_{\,\,\phiup}=\{\alpha_1{+}\alpha_2\,{=}\,e_{1}{-}
 e_{3},\alpha_2\,{=}\,e_{2}{-}e_{3},\alpha_2{+}\alpha_3\,{=}\,e_{2}{-}e_{4}\},\\
 &\Qc_{\;\phiup}\cap\sigmaup({\Qc}^{-n}_{\phiup})
 =\{\alpha_3\,{=}\,e_{3}{-}e_{4},-\alpha_2\,{=}\,e_{3}{-}e_{2},\alpha_1\,{=}\,e_{1}{-}e_{2}\}.
\end{align*}
We obtain that $(\gt^{\sigmaup},\qt_{\phiup})$
is fundamental and finitely nondegenerate of order~$2$,  since $\ktt^{\sigmaup}_{\phiup}(e_{3}{-}e_{2})\,{\geq}2$
and equals $2$ because 
\begin{equation*}
 e_{1}-e_{4}=(e_{3}-e_{2})+(e_{1}{-}e_{3})+(e_{2}{-}e_{4}).\;\;
% e_{1}-e_{4}=(e_{1}-e_{2})+(e_{2}{-}e_{4}).
\end{equation*}
\end{exam}

\begin{rmk}
Since $\xiup_{\phiup}(\alpha){\leq}-1$ for all $\alpha\,{\in}\Qc^{-n}_{\,\,\phiup},
$ 
 if $\beta\,{\in}\,\Qc_{\;\phiup}\,{\cap}\,\sigmaup({\Qc}^{-n}_{\,\,\phiup})$ and \hfill\par\noindent
 $\ktt^{\sigmaup}_{\phiup}(\beta){<}{+}\infty,$ then 
\begin{equation}\label{eq2.10}
 \ktt^{\sigmaup}_{\phiup}(\beta)\leq 1+\xiup_{\phiup}(\beta).
\end{equation}
\end{rmk}
\begin{cor}\label{c4.6}
 If $\alpha\,{\in}\, \Qc^{\,r}_{\;\phiup}\,{\cap}\,\sigmaup(\Qc_{\;\phiup}^{-n}),$ 
 then $\ktt^{\sigmaup}_{\phiup}(\alpha)=1$   or~$\ktt^{\sigmaup}_{\phiup}
 (\alpha)={+}\infty.$\qed
\end{cor} 
We obtain also a useful criterion of finite nondegeneracy (cf. \cite[Thm.6.4]{AlMeNa2006b})
\begin{prop} \label{p2.9}
The parabolic  CR  algebra $(\gt^{\sigmaup},\qt_{\phiup})$
is 
finitely nondegenerate if and only if 
\begin{equation}\label{e2.12}
 \forall \beta\,{\in}\,\Qc_{\;\phiup}\cap\sigmaup({\Qc}^{-n}_{\;\phiup})\;\;\exists\, 
 \alpha\in\sigmaup({\Qc}_{\;\phiup})
 \cap\Qc^{-n}_{\;\phiup}\;\;\text{such that}\;\; \beta+\alpha\in\sigmaup(\Qc_{\;\phiup}^{-n}).
\end{equation}\qed
\end{prop} 
We reharse the result contained in \cite[Cor.2.19]{MaMeNa2021},
in a form which is suitable for our purposes.
\begin{thm}\label{th.5.1} 
The closed orbit is either holomorphically 
degenerate or 
is finitely nondegenerate with Levi order $\ktt$ less or equal than two. 
\end{thm} 
\begin{proof}
Let $(\gt^{\sigmaup},\qt_{\phiup})$ be a parabolic  CR  algebra of the 
closed orbit,
 $\phiup$ being the set of crossed roots in a  
 cross-marked Satake diagram. 
 Since all roots $\beta$ 
 in ${\Qc}^{\,n}_{\,\,\phiup}$ are positive, 
by  \eqref{Adapted Weyl}, if $\beta\,{\in}\,{\Qc}^{\,n}_{\,\,\phiup}{\setminus}\Rad_{\,
\bullet},$ 
then its conjugate 
$\sigmaup({\beta})$ is positive, and hence has $\xiup_{\phiup}(\sigmaup(\beta))\,{\geq}\,0$, i.e.
\begin{equation}\label{asterisco}
 \xiup_{\phiup}(\beta)\geq{0},\;\;\forall \beta\,{\in}\sigmaup({\Qc}^{\,n}_{\,\,\phiup})
 \backslash\Rad_{\,\bullet}^{\sigmaup}.
\end{equation}
Let $\beta\,{\in}\,\Qc_{\;\phiup}\,{\cap}\,\sigmaup({\Qc}_{\;\phiup}^{-n}).$ If $
\beta\,{\notin}\,\Rad_{\,\bullet}^{\sigmaup},$
then $\xiup_{\phiup}(\beta)\,{=}\,0$ by \eqref{asterisco} and hence, by Cor.
~\ref{c4.6},
$\ktt^{\sigmaup}_{\phiup}(\beta)$ is either $1$ or ${+}\infty.$ \par
If
$\ktt^{\sigmaup}_{\phiup}(\beta)$ is an integer $k{>}1,$ 
then $\beta\,{\in}\,\Rad_{\,\bullet}^{\sigmaup}.$ 
Let $(\alpha_{1},\hdots,\alpha_{k})$ 
be a sequence satisfying \eqref{e4.3} and thus the conditions in 
Prop.\ref{p4.3}.\par
Since ${\Qc}^{-n}_{\;\phiup}\,{\cap}\,\sigmaup({\Qc}^{-n}_{\;\phiup})\,{\cap}\,\Rad_{\,
\bullet}\,{=}\,\emptyset,$
there is at least one root $\alpha_{i}$ which does not belong to $
\Rad_{\,\bullet}^{\sigmaup}.$
Again by $(iii)$ of Proposition \ref{p4.3} we can assume it is $\alpha_{1}.$ Then $\beta{+}
\alpha_{1}$
belongs to $(\Qc_{\;\phiup}\,{\cap}\,\sigmaup({\Qc}^{-n}_{\,\,\phiup})){\backslash}
\Rad_{\,\bullet}^{\sigmaup}$ and therefore,
by the first part of the proof, $\ktt^{\sigmaup}_{\phiup}(\beta{+}\alpha_{1})
{=}1.$
This implies that $k{=}2.$ The proof is complete. 
\end{proof}
\begin{rmk}
%We notice that 
Theorem~\ref{th.5.1}  is a special case of a more general result 
proved in \cite{MaMeNa2021} which state that all 
finitely nondegenerate parabolic CR algebras have 
Levi order less or equal $3$.
A special class, generalizing minimal orbits, and called of 
\emph{minimal type} are proved there to  have also $\ktt{\leq}2$.
\end{rmk}
\begin{prop}\label{p4.10}
 Let $(\gs,\qt_{\phiup})$ be a  
  parabolic $CR$ algebra, corresponding  to a minimal orbit. Then 
\begin{equation}\label{e4.7}
 \Qc_{\;\phiup}\,{\cap}\,\sigmaup(\Qc_{\;\phiup}^{-n})=(\Qc^{n}_{\;\phiup}\cap\Rad
 ^{\sigmaup}
 _{\;\bullet})
 \cup (\Qc^{r}_{\;\phiup}\,{\cap}\,\sigmaup(\Qc_{\;\phiup}^{-n})).
\end{equation}
\par
  If $(\gs,\qt)$ is finitely nondegenerate,  then its order of nondegeneracy is 
\begin{itemize}
 \item $1$, if $\Qc^{n}_{\;\phiup}\cap\Rad
 ^{\sigmaup}
 _{\;\bullet}\,{=}\,\emptyset$;
 \item $\sup\{\ktt^{\sigmaup}_{\phiup}(\beta)\,{\mid}\,\beta\,{\in}\,\Qc^{n}_{\;\phiup}\,{\cap}\,\Rad
 %^{\sigmaup}
 _{\;\bullet}\}$,
   if $\Qc^{n}_{\;\phiup}\,{\cap}\,\Rad
   ^{\sigmaup}
   _{\;\bullet}\,{\neq}\,\emptyset$
\end{itemize}
\end{prop} 
\begin{proof}
We assume, as we can, that $(\gs,\qt_{\phiup})$ is described by a cross-marked Satake diagram.
In particular, complex positive roots have a positive $\sigmaup$-conjugate. This implies
that all complex roots that do not belong to $\sigmaup(\Qc_{\;\phiup})$ are negative and hence,
if they are in $\Qc_{\;\phiup}$, they 
belong to $\Qc_{\;\phiup}^{r}$. On the other hand,
imaginary roots not belonging to $\sigmaup(\Qc_{\;\phiup})$ belong
to $\Qc^{n}_{\;\phiup}$.
\par
The statement about the order of nondegeneracy follows from Corollary\,\ref{c4.6}.
\end{proof}

\section{$\Gf^{\sigmaup}$-equivariant  CR  fibrations}\label{s5}
In this section we describe  the canonical 
$G^{\sigmaup}$-equivariant  CR-fibrations of the closed orbits 
and their characterization 
in terms of  
associated \emph{cross-marked  Satake diagrams}.  
Together with  Lemma~\ref{l1.3}, this will give us effective criteria of finite nondegeneracy.
\par

\begin{rmk}
Let $(\gs,\qt)$ be a  CR  algebra, $\at^{\sigmaup}$ an ideal of $\gs$ and $\at$ its complexification. 
Set 
\begin{equation}\begin{cases}
 \ft=(\qt+\at)\cap(\sigmaup(\qt)+\at),\\
 \ft^{\sigmaup}=\ft\cap\,\gs,\\
 \qt'=\qt\cap(\sigmaup(\qt)+\at).
 \end{cases}
\end{equation}
Since 
\begin{equation*}
 \qt+(\qt+\at)\cap(\sigmaup(\qt)+\at)=\qt+\at,
\end{equation*}
the following fibration is CR: 
\begin{equation} \label{e4.6}
\begin{CD}
 0 @>>> (\ft^{\sigmaup},\qt')@>>> (\gs,\qt) @>>> (\gs,\qt+\at)@>>> 0. 
\end{CD}
\end{equation}
\end{rmk}
\vspace{1cm}
\par

If $\psiup\,{\subseteq}\,\phiup\,{\subseteq}\,\Bz$, then $\qt_{\phiup}\,{\subseteq}\,\qt_{\psiup}$ and we obtain
a $\gt^{\sigmaup}$-equivariant fibration 
\begin{equation} \label{e5.3}
\begin{CD}
 0 @>>> (\ft_{\psiup}^{\sigmaup},\qt_{\phiup,\psiup})@>>> (\gs,\qt_{\phiup}) @>>> (\gs,\qt_{\psiup}) @>>> 0
\end{CD}
\end{equation}
where the \textit{typical fiber} $( \ft^{\sigmaup}_{\;\psiup}, \qt_{\phiup,\psiup})$ is described by 
\begin{equation} 
\begin{cases}
 \ft_{\psiup}=\qt_{\psiup}\cap\sigmaup(\qt_{\psiup}),\\
 \ft_{\psiup}^{\sigmaup}=\qt_{\psiup}\cap\gs,\\
 \qt_{\phiup,\psiup}=\qt_{\phiup}\cap\sigmaup(\qt_{\psiup}).
\end{cases}
 \end{equation}
 \par
 We recall that \eqref{e5.3} is a  CR  fibration  iff 
\begin{equation}
 \qt_{\psiup}=\qt_{\psiup}\cap\sigmaup(\qt_{\psiup})+\qt_{\phiup}.
\end{equation}
In terms of roots, this equality can be rewritten by 
\begin{align*}
 &\Qc_{\;\psiup}=\left(\Qc_{\;\psiup}\cap\sigmaup(\Qc_{\;\phiup})\right)\cup\Qc_{\;\phiup}\\
 &\quad
 \Leftrightarrow \Rad\setminus\Qc^{-n}_{\;\psiup}=
 \left[(\Rad\setminus\Qc^{-n}_{\;\psiup})\cap\sigmaup(\Rad\setminus\Qc^{-n}_{\;\psiup})\right]
 \cup(\Rad\setminus\Qc^{-n}_{\,\,\phiup})\\
 &\quad
 \Leftrightarrow  \Rad\setminus\Qc^{-n}_{\;\psiup}=
 \left[\Rad\setminus\left(\Qc^{-n}_{\;\psiup}\cup\sigmaup(\Qc^{-n}_{\;\psiup})\right)\right]
 \cup(\Rad\setminus\Qc^{-n}_{\,\,\phiup})\\
 &\quad
 \Leftrightarrow \Qc^{-n}_{\;\psiup}=(\Qc^{-n}_{\;\psiup}\cup\sigmaup(\Qc^{-n}_{\;\psiup}))\cap\Qc^{-n}_{\,\,\phiup}\\
 &\quad
 \Leftrightarrow \Qc^{n}_{\;\psiup}=(\Qc^{n}_{\;\psiup}\cup\sigmaup(\Qc^{n}_{\;\psiup}))\cap\Qc^{n}_{\,\,\phiup}\\
 &\quad
 \Leftrightarrow \Qc^{n}_{\,\,\phiup}\cap\sigmaup(\Qc^{n}_{\;\psiup})\subseteq\Qc^{n}_{\;\psiup}.
\end{align*}
We summarize the above discussion by the following proposition.
\begin{prop}
 Let $C$ be an adapted Weyl chamber for the parabolic  CR  algebra $(\gs,\qt_{\phiup})$, described by
 the subset $\phiup$ of the basis $\Bz\,{=}\,\Bz(C)$ of simple positive roots of $\Rad^{+}(C)$. 
 If there is a $\gs$-equivariant   CR  algebra homomorphism of $(\gs,\qt_{\phiup})$
 onto a  CR  algebra $(\gs,\qt')$, then this is also parabolic and $\qt'\,{=}\,\qt_{\psiup}$ for a subset
 $\psiup$ of $\Bz$ with $\psiup\,{\subseteq}\,\phiup\,{\subseteq}\,\Bz$. The corresponding fibration \eqref{e5.3}
 is  CR  if and only if 
\begin{equation}\label{e5.4} \vspace{-18pt}
\Qc^{n}_{\;\phiup}\cap\sigmaup(\Qc^{n}_{\;\psiup})\subseteq\Qc^{n}_{\;\psiup}.
\end{equation}
\qed
\end{prop} 
In general, $\gt_{\psiup}^{\sigmaup}$ and $\gt_{\psiup}$ may not be semisimple. However,
 both $\ft_{\psiup}$ and $\qt_{\phiup,\psiup}$ are \textit{regular} subalgebras of $(\gt,\hg)$. 
 With 
\begin{equation}
 \Fc_{\psiup}^{r}{=}\Qc^{r}_{\;\psiup}{\cap}\,\sigmaup(\Qc^{r}_{\;\psiup}),\;
 \Fc_{\psiup}^{n}{=}(\Qc_{\;\psiup}{\cap}\,\sigmaup(\Qc^{n}_{\;\psiup})){\cup}(\Qc^{n}_{\;\psiup}{\cap}
 \,\sigmaup(\Qc_{\;\psiup})),\;  \Fc_{\psiup}{=} \Fc_{\psiup}^{r}{\cup}\Fc^{n}_{\;\psiup},
\end{equation}
we get a Levi-Chevalley decomposition of $\ft_{\psiup}$:
\begin{equation}
 \ft_{\psiup}\,{=}\,\ft_{\psiup}^{r}\oplus\ft^{n}_{\;\psiup},\;\;\text{with}\;\; 
\begin{cases}
 \ft_{\psiup}^{r}=\hg\oplus{\sum}_{\alpha\in \Fc_{\psiup}^{r}}\gt^{\alpha} &\text{(reductive Levi factor)},\\
 \ft^{n}_{\;\psiup}={\sum}_{\alpha\in\Fc^{n}_{\;\psiup}}\gt^{\alpha} &\text{(nilradical)}.
\end{cases}
\end{equation}
\par
Analogously, with 
\begin{equation}
 \Qc_{\phiup,\psiup}{=}\Qc_{\;\phiup}{\cap}\sigmaup(\Qc_{\;\psiup}),\; 
 \Qc^{r}_{\phiup,\psiup}{=}\,\Qc^{r}_{\phiup}{\cap}\sigmaup(\Qc^{r}_{\;\psiup}),
 \;
 \Qc^{n}_{\;\phiup,\psiup}{=}\,(\Qc^{n}_{\;\phiup}{\cap}\sigmaup(\Qc_{\;\psiup}))
 {\cup}(\Qc_{\;\phiup}{\cap}\sigmaup(\Qc^{n}_{\;\psiup})),
\end{equation}
we get a Levi-Chevalley decomposition of $\qt_{\phiup,\psiup}$:
\begin{equation}
 \qt_{\phiup,\psiup}\,{=}\,\qt_{\phiup,\psiup}^{r}\oplus\qt^{n}_{\phiup,\psiup},\;\text{with}\; 
\begin{cases}
 \qt_{\phiup,\psiup}^{r}=\hg\oplus{\sum}_{\alpha\in\Qc_{\phiup,\psiup}^{r}}\gt^{\alpha} &\text{(reductive Levi factor)},\\
 \qt^{n}_{\phiup,\psiup}={\sum}_{\alpha\in\Qc^{n}_{\;\phiup\psiup}}\gt^{\alpha} &\text{(nilradical)}.
\end{cases}
\end{equation} 
\begin{lem} Assume that $C$ has property \eqref{Adapted Weyl} for $\sigmaup$. 
Then 
\begin{equation}
 \Fc^{n}_{\;\psiup}\subseteq\Rad^{+}(C){\setminus}\Rad_{\,\;\bullet}^{\sigmaup}\subseteq\Qc_{\;\phiup}
 \;\Longrightarrow\; \ft_{\psiup}^{n}\subseteq\qt_{\phiup}.
\end{equation}
\end{lem} 
\begin{proof}
 Let us check that $\Qc^{n}_{\;\psiup}\,{\cap}\,\sigmaup(\Qc_{\;\psiup})\,{\cap}\,\Rad_{\,\;\bullet}^{\sigmaup}\,{=}\,\emptyset$.
 Assume that $\alpha\,{\in}\,\Qc^{n}_{\;\psiup}\,{\cap}\,\Rad_{\,\;\bullet}^{\sigmaup}$.
 Then ${-}\alpha\,{\in}\,\sigmaup(\Qc^{n}_{\;\psiup})$ shows that $\alpha\,{\in}\,\sigmaup(\Qc_{\;\psiup}^{-n})$,  i.e.  that
 $\alpha\,{\notin}\,\sigmaup(\Qc_{\;\psiup})$.  Since $\Qc_{\;\psiup}^{n}$ is contained in $\Rad^{+}(C)$
 and we assumed that $C$ satisfies
 \eqref{Adapted Weyl}, we obtain that also
 $\sigmaup(\Qc^{n}_{\;\psiup}\,{\cap}\,\sigmaup(\Qc_{\;\psiup}))
 \,{=}\,\Qc_{\;\psiup}\,{\cap}\,\sigmaup(\Qc^{n}_{\;\psiup})\,{\subseteq}\,
 \Rad^{+}(C)$. The proof is complete.
\end{proof}

\begin{lem}
The set $\Fc^{r}_{\psiup}$ is a subsystem of roots of $\Rad$,
\begin{equation}
  \Bz_{\psiup}=\{\alpha\in\Bz\mid \alpha,\sigmaup(\alpha)\notin\Qc_{\;\psiup}^{n}\}
\end{equation}
a basis of simple positive roots of $\Fc^{r}_{\psiup}$ and $\Qc_{\phiup,\psiup}{\cap}\,\Fc^{r}_{\psiup}$
a parabolic subset of~$\Fc^{r}_{\phiup}$.\qed
\end{lem}
\begin{prop}\label{p5.5}
 Let $\psiup\,{\subseteq}\,\phiup\,{\subseteq}\,\Bz$ and assume that 
 $(\gt^{\sigmaup},\qt_{\phiup})\,{\to}\,(\gt^{\sigmaup},\qt_{\psiup})$
 is a CR fibration. A necessary and sufficient condition for the typical fiber being totally complex is that 
\begin{equation}\label{e5.13}
 \qt_{\phiup}+\sigmaup(\qt_{\phiup})= \qt_{\psiup}+\sigmaup(\qt_{\psiup}).
\end{equation}
\end{prop} 
\begin{proof}
Indeed we have 
 \begin{align*}
 \qt_{\phiup,\psiup}\,{+}\,\sigmaup(\qt_{\phiup,\psiup})\,{=}\,\ft_{\psiup}\;\Leftrightarrow\;
 \qt_{\psiup}\cap\sigmaup(\qt_{\phiup})+\qt_{\phiup}\cap\sigmaup(\qt_{\psiup})=\qt_{\psiup}\cap\sigmaup(\qt_{\psiup}).
 \end{align*} 
 By the assumption that the fibration is CR we get 
\begin{equation*}
 \qt_{\psiup}=\qt_{\phiup}+\qt_{\psiup}\cap\sigmaup(\qt_{\psiup})=
\qt_{\phiup}+ \qt_{\psiup}\cap\sigmaup(\qt_{\phiup})+\qt_{\phiup}\cap\sigmaup(\qt_{\psiup})
\subset \qt_{\phiup}+\sigmaup(\qt_{\phiup})
\end{equation*}
and this implies that the right hand side is contained in the left hand side of
\eqref{e5.13}. Since the opposite inclusion is obvious, the proof is complete.
 \end{proof}
 \begin{rmk} Condition \eqref{e5.13} is equivalent to any of 
\begin{gather}
\tag{\ref{e5.13}a}
 \qt_{\phiup}\subseteq\qt_{\psiup}\subseteq\qt_{\phiup}+\sigmaup(\qt_{\phiup}),\\
 \tag{\ref{e5.13}b}
 \qt_{\phiup}+\qt_{\psiup}\cap\sigmaup(\qt_{\phiup})=\qt_{\psiup},\\
 \tag{\ref{e5.13}c}
 \qt_{\phiup}\cap \sigmaup(\qt_{\psiup})\subseteq\qt_{\psiup}.
\end{gather}
 In terms of roots, we can translate \eqref{e5.13} into 
\begin{equation} (\ft^{\sigmaup}_{\;\psiup},\qt_{\phiup,\psiup})\,\;\text{is totally complex}\; \Leftrightarrow\;\;
 \Qc_{\;\psiup}\cup\sigmaup(\Qc_{\;\psiup})= \Qc_{\;\phiup}\cup\sigmaup(\Qc_{\;\phiup}).
\end{equation}
\end{rmk}
\par\smallskip
We have the following criterion on its cross-marked Satake diagram 
to establish whether a minimal orbit is \textit{fundamental} 
(see Def.~\ref{fundamental}):
\begin{prop}[{\cite[Th. 9.1]{AlMeNa2006}}]\label{p5.7}
The effective parabolic  CR  algebra  $(\gt^{\sigmaup},\qt_{\phiup})$ 
of a closed orbit is \emph{fundamental}  
if and only if its corresponding cross-marked Satake diagram $(\mathcal{S},\phiup)$ has the property
\begin{equation}\vspace{-18pt}
\alpha\in\phiup\setminus\Rad_{\;\bullet}^{\sigmaup}\Rightarrow\epi(\alpha)\notin\phiup .
\end{equation}\qed
\end{prop} 

\begin{ntz}
Given a Satake diagram $\mathcal{S}$ and a subset of simple roots $\phiup\subset\Bz$, we will use the notation:
\begin{itemize}
\item Let $\alpha$ be any positive root. We denote by $\phiup^{\circ}(\alpha)$
the connected component of $\supp(\alpha)$ in $(\Bz{\setminus}\phiup)\,{\cup}\,\supp(\alpha)$.

\item The \emph{exterior boundary} $\de\etaup$ of a given a subset $\etaup$ 
of $\mathcal{B}$  is the set 
\begin{equation*}
 \de\etaup=\{\alpha\in\Rad\mid\exists\,\beta\,{\in}\,\etaup\;\text{s.t.}\;\alpha{+}\beta\,{\in}\,\Rad\}.
\end{equation*}
\end{itemize}
\end{ntz}

\begin{prop}[{\cite[Th.7.10]{AlMeNa2006}}]\label{p5.8}
Let $\psiup\,{\subseteq}\,\phiup\,{\subseteq}\,\Bz$.
A necessary and sufficient condition for \eqref{e5.3} to be 
a  $\gt^{\sigmaup}$-equivariant CR fibration is that for each
$\alpha\,{\in}\,\phiup{\setminus}\psiup$ either one of the following conditions holds:
\begin{align*}
(1) &\quad \psiup^{\circ}(\alpha)\subset\Rad_{\;\bullet}^{\sigmaup};\\
(2) &\quad \psiup^{\circ}(\alpha)\nsubset\Rad_{\;\bullet}^{\sigmaup}, \;\; 
\epsilon_C\left(\psiup^{\circ}(\alpha)\setminus\Rad_{\;\bullet}^{\sigmaup}\right)\cap
\psiup=\emptyset,\;\;\text{and}\;\; \de(\psiup^{\circ}(\alpha))\cap
\Rad_{\;\bullet}^{\sigmaup}=\emptyset.
\end{align*} \qed
\end{prop}

\begin{rmk}\label{rmk.6.6.}
Let $(\gt^{\sigmaup},\qt_{\phiup})$ be the effective parabolic  CR  
algebra of a closed orbit.
By Prop.~\ref{p5.5},  $(\gt^{\sigmaup},\qt_{\phiup})$ 
is finitely degenerate if and only if we can find a
$\psiup\,{\subsetneqq}\,\phiup$,  satisfying conditions $(1)$ and $(2)$  of  
Proposition~\ref{p5.8}, such that  \par\centerline{$\qt_{\psiup}\subset\qt_{\phiup}+\sigmaup({\qt}_{\phiup})$.}
\end{rmk}

\begin{thm}[{\cite[Th.11.5]{AlMeNa2006}}]\label{t5.10}
Let $(\gt^{\sigmaup}, \qt_{\phiup})$ be a simple fundamental effective 
parabolic minimal  CR  algebra and assume that it is not totally complex.  
Let $\Pi\subseteq\phiup$ be the set of simple roots $\alpha\in\phiup$  
that satisfy either one of:\begin{enumerate}
\item $\phiup^{\circ}(\alpha)\subset\Rad_{\;\bullet}^{\sigmaup}$;
\item $\left(\phiup^{\circ}(\alpha)\cup\de\phiup^{\circ}(\alpha)\right)\cap
\Rad_{\bullet}^{\sigmaup}=\emptyset$ and $\epsilon_C\left(\phiup^{\circ}(\alpha)\right)\cap
\phiup=\emptyset.$
\end{enumerate}
Then $(\gt^{\sigmaup}, \qt_{\phiup})$ is 
finitely nondegenerate
if and only if $\Pi =\emptyset$.\par
In general, if we set $\psiup=\phiup\setminus\Pi$, then \eqref{e5.3}
is a $\gt^{\sigmaup}$-equivariant  CR  fibration with fundamental finitely nondegenerate base
and totally  complex typical fiber.\qed
\end{thm}
\begin{thm}\label{t5.11}
Let $(\gt^{\sigmaup}, \qt_{\phiup})$ be a finitely nondegenerate fundamental
minimal effective parabolic $CR$ algebra, 
with $\gs$ simple. \par
If 
$\alpha\,{\in}\,\Qc_{\;\phiup}\,{\cap}\,\sigmaup(\Qc_{\;\phiup}^{-n})$ has $\ktt^{\sigmaup}_{\phiup}(\alpha)\,{=}\,2$,
then $\alpha\,{\in}\,\Qc^{n}_{\;\phiup}\,{\cap}\,\Rad^{\sigmaup}_{\;\bullet}$ and
$\phiup\,{\subseteq}\,\supp(\alpha)$.\end{thm} 
\begin{proof} Since $C$ is an $S$-chamber, $\Qc_{\;\phiup}\,{\cap}\,\sigmaup(\Qc_{\;\phiup}^{-n})$ is
the disjoint union 
\begin{equation*}
 \Qc_{\;\phiup}\,{\cap}\,\sigmaup(\Qc_{\;\phiup}^{-n})
=\left(\Qc^{r}_{\;\phiup}\,{\cap}\,\sigmaup(\Qc^{-n}_{\phiup})\right)\cup
\left(\Qc^{n}_{\;\phiup}\,{\cap}\,\Rad^{\sigmaup}_{\;\bullet}\right),
\end{equation*}
the first set in the right hand side being contained in $\Rad^{-}(C)$, the second in~$\Rad^{+}(C)$.\par
Fix any $\alpha\,{\in}\,\Qc_{\;\phiup}\,{\cap}\,\sigmaup(\Qc_{\;\phiup}^{-n})$.
We want to prove that, if either $\alpha\,{\in}\,\Qc^{r}_{\;\phiup}\,{\cap}\,\sigmaup(\Qc^{-n}_{\phiup})$  or
$\alpha\,{\in}\,\Qc^{n}_{\;\phiup}\,{\cap}\,\Rad^{\sigmaup}_{\;\bullet}$ and 
$\phiup\,{\not\subset}\,\supp(\alpha),$ then
$\ktt_{\phiup}^{\sigmaup}(\alpha)\,{=}\,1$. \par
By Corollary\,\ref{c4.6}, we get $\ktt_{\phiup}^{\sigmaup}(\alpha)\,{=}\,1$ when 
$\alpha\,{\in}\,\Qc^{r}_{\;\phiup}\,{\cap}\,\sigmaup(\Qc^{-n}_{\phiup})$.
\par 
Suppose 
$\alpha\,{\in}\,\Qc^{n}_{\;\phiup}\,{\cap}\,\Rad^{\sigmaup}_{\;\bullet}$ and let
 $\psiup$ be the connected component of $\supp(\alpha)$ in~$\Bz^{\sigmaup}_{\bullet}$.\par 
Assume that $\psiup{\setminus}\supp(\alpha)$ contains an element $\alpha'$ of $\phiup$.
 By Theorem\,\ref{t5.10}, 
$\phiup^{\circ}(\alpha)\,{\not\subset}\,\Bz^{\sigmaup}_{\bullet}$. Thus we can find 
a simple path $(\alpha_{0},\hdots,\alpha_{p})$ in $\phiup^{\circ}(\alpha){\setminus}\supp(\alpha)$
such that 
$\alpha_{0}\,{\notin}\,\Bz_{\bullet}^{\sigmaup}$, $\alpha_{i}\,{\in}\,\psiup$ for $0{<}i{\leq}p$,
joining $\alpha_{0}$ to $\supp(\alpha)$, i.e. satisfying 
\begin{equation*} \tag{$*$}
\begin{cases}
 (\alpha_{i}\,|\,\alpha_{j})= 0, & \text{for $|i-j|>1$},\\
 (\alpha_{i-1}\,|\,\alpha_{i})<0, & \text {for $1\,{\leq}\,i\,{\leq}\,p$},\\
 (\alpha_{p}\,|\,\alpha)<0.\end{cases}
\end{equation*}
Let $\gamma\,{=}\,{-}{\sum}_{i=0}^{p}{\alpha_{i}}$ and $\beta\,{=}\,\sigmaup(\gamma)$. Then $\beta$
is a negative root 
in $\sigmaup(\Qc^{r}_{\;\phiup})\,{\cap}\,\Qc^{-n}_{\;\phiup}$,
because
$\alpha'$, belonging to the support of $\sigmaup(\alpha_{0})$,
also belongs to the support of $\beta$.
The inequality
\begin{equation*}\tag{$**$}
 (\alpha\,|\,\beta)=(\sigmaup(\alpha)\,|\,\gamma)\,{=}\,{-}(\alpha,\gamma)=(\alpha\,|\,\alpha_{p})<0
\end{equation*}
implies that $\alpha\,{+}\beta$ is a negative root, containing $\alpha'$ in its support, while
$\sigmaup(\alpha{+}\beta)\,{=}\,{-}\alpha\,{+}\,\gamma$ contains $\alpha$ in its support.
Hence $\alpha\,{+}\,\beta\,{\in}\,\Qc^{-n}_{\;\phiup}\,{\cap}\,\sigmaup(\Qc^{-n}_{\;\phiup})$
and this shows that $k_{\phiup}^{\sigmaup}(\alpha)\,{=}\,1$. \par\smallskip
If $\phiup\,{\cap}\,\psiup\,{\subseteq}\,\supp(\alpha)$, but $\phiup\,{\not\subset}\supp(\alpha)$, 
then we can find a simple path \hfill \par
\noindent $(\alpha_{0},\hdots,\alpha_{p})$ in $\Bz$,
satisfying $(*)$, with $\alpha_{0}\,{\in}\,\phiup$, $\alpha_{i}\,{\notin}\,\phiup$ for $1{\leq}i{\leq}p$.
\par
If we can find such a sequence with
$\alpha_{0}\,{\in}\,\Bz^{\sigmaup}_{\bullet}$, then $p{\geq}1$ and 
$\gamma\,{=}\,{-}{\sum}_{i=1}^{p}\alpha_{i}$ is a not imaginary negative root in $\Qc^{r}_{\;\phiup}$. 
Set 
$\beta\,{=}\,\sigmaup(\gamma)$. \par
Since $\{\alpha_{0}\}\,{\cup}\,\supp(\alpha)\,{\subset}\,\supp(\beta)$, we get
$\beta\,{\in}\,\sigmaup(\Qc^{r}_{\;\phiup})\,{\cap}\,\Qc^{-n}_{\;\phiup}$. Also $(**)$ holds true,
so that $\alpha\,{+}\,\beta\,{\in}\,   \Qc^{-n}_{\;\phiup}\,{\cap}\,\sigmaup(\Qc^{-n}_{\;\phiup})$
and hence $\ktt_{\phiup}^{\sigmaup}(\alpha)\,{=}\,1$. \par
Otherwise, we consider the set $\hat{\phiup}\,{=}\,\phiup\,{\cup}\,\epi(\phiup{\setminus}\Bz^{\sigmaup}_{\bullet})$.
Let $\hat{\psiup}$ be the connected component of $\psiup$ in $(\Bz{\setminus}\hat{\phiup})\,{\cup}\,\psiup$.
This set is $\epi$-invariant and hence also its exterior boundary is $\epi$-invariant.
In particular, it contains an element $\alpha_{0}$ of $\phiup{\setminus}\Bz^{\sigmaup}_{\bullet}$ and
its symmetrical 
$\epi(\alpha_{0})$, which, by Proposition\,\ref{p5.7},  
does not belong to $\phiup$. 
We take simple paths $(\alpha_{0},\hdots,\alpha_{p})$ and $(\alpha_{p+1},\hdots,\alpha_{q})$, 
with $\alpha_{i}\,{\in}\,\Bz{\setminus}\phiup$ for $i\,{\neq}\,0,q$,
joining 
$\supp(\alpha)$ to $\alpha_{0}$ and to $\alpha_{q}{=}\epi(\alpha_{0})$, respectively.
This means
that $(*)$ holds true and moreover 
\begin{equation*} 
\begin{cases}
 (\alpha_{i}\,|\,\alpha_{j})=0, & \text{for $1{\leq}i{<}i{+}2{\leq}j{\leq}q,$}\\
 (\alpha_{p}\,|\,\alpha_{p+1})=0,\\
 (\alpha_{p+1}\,|\,\alpha)<0.
\end{cases}
\end{equation*}
Then $\gamma\,{=}\,{-}{\sum}_{i=p+1}^{q}\alpha_{i}\,{\in}\,\Qc^{r}_{\;\phiup}$ and
$\beta\,{=}\,\sigmaup(\gamma)\,{\in}\,\sigmaup(\Qc^{r}_{\;\phiup})\,{\cap}\,\Qc_{\;\phiup}^{-n}$,
since $\beta$ is a negative root whose support contains $\alpha_{0}$. Since $(**)$ holds true,
$\alpha\,{+}\,\beta\,{\in}\,\Qc_{\;\phiup}^{-n}\,{\cap}\,\sigmaup(\Qc_{\;\phiup}^{-n})$,
proving that $\ktt^{\sigmaup}_{\phiup}(\alpha)\,{=}1$. \par
The proof is complete.\end{proof}

\section{Finitely nondegenerate closed orbits}\label{s6}
Putting together the results of the previous sections, 
we obtain  a complete classification of
finitely Levi-nondegenerate closed orbits of real forms in complex flag manifolds
and may compute their order of Levi nondegeneracy.
Since, by  Remark\,\ref{prodotti},
they are Cartesian products of orbits of simple real Lie groups,
we can restrain to the case of a simple $\Gf^{\sigmaup}$.
We first consider the case where $\Gf$ is also a simple Lie group.
\begin{thm} Let $(\gs,\qt_{\phiup})$ be the CR algebra of a closed orbit and assume that
$\gs$ is simple and of the real type\footnote{This means that its complexification $\gt$ is simple.}.
If $(\gs,\qt_{\phiup})$ is finitely nondegenerate, then $\gs$ in not compact and
the subset $\phiup$ is one of those listed in Tables~\ref{finite nnndg} and ~\ref{finite nnndg ex},
where also its  Levi order  is indicated.
\end{thm}
\begin{proof}
By Theorem\,\ref{t3.4},  closed orbits in complex flag manifolds are in a one-to-one correspondence with
cross-marked Satake diagrams $(\mathcal{S},\phiup)$.  Connected 
Satake diagrams classify simple real Lie algebras of the real type
(see e.g.~\cite{Ar1962}) and properties of the minimal parabolic 
CR algebras defined by different $\phiup$'s
can be discussed by utilizing 
the results of the previous section. \par
For the classification of simple real forms of the real type 
we use the notation of \cite[Table VI, pp.532-534]{He1978}.\par 
Real types  $\textsc{A\,I}$,  $\textsc{C\,I}$, $\textsc{E\,I}$, $\textsc{E\,V}$, $\textsc{E\,VIII}$,
$\textsc{F\,I}$, $\textsc{G}$ and the split $\textsc{B\,I}$, $\textsc{D\,I}$, being split real forms,
yield totally real $(\gs,\qt_{\phiup})$ for all choices of $\phiup\,{\subseteq}\,\Bz$. 
\par 
Types $\textsc{A\,{II}}$ and $\textsc{D\,{II}}$, with the $\phiup$'s  
listed Table~\ref{complex ones}, %in Theorem\,\ref{t3.6}, and 
and
 all compact forms, for all $\phiup\,{\subseteq}\,\Bz$,  yield
totally complex CR algebras $(\gs,\qt_{\phiup})$ which are associated to the corresponding full 
complex flag manifold.\par
By using 
Proposition\,\ref{p5.7}, we can recognize the  
$\phiup$'s  yielding fundamental CR algebras and by 
Theorem\,\ref{t5.10} 
 the $\phiup$'s for which $(\gs,\qt_{\phiup})$ is finitely nondegenerate.
\par
By Corollary\,\ref{c4.6} we know that a finitely nondegenerate 
minimal parabolic 
$(\gs,\qt_{\phiup})$ has Levi order less or equal $2$
and by 
Theorem\,\ref{t5.11} that, when $\gs$ is simple and of the real type,
only those with $\phiup$ contained in the support of an imaginary root may
have Levi order $2$. Hence we will restrain the computation of this invariant to
these cases.
%%%%%%%%%%%%%%%
\par \smallskip\noindent
$\boxed{\textsc{$\gs$ of type\,\,AII}}$ We have 
$\gs\,{\simeq}\,\slt_{p}(\Hb)$, for some integer $p\,{\geq}\,2$.\par\vspace{3pt}
 The $\phiup$'s for which
$(\gs,\qt_{\phiup})$ is fundamental 
are those with $\phiup\,{\subseteq}\,\Bz^{\sigmaup}_{\bullet}$. 
Those with $\phiup$ equal to $\{\alpha_{1}\}$ or $\{\alpha_{2p-1}\}$ are totally complex.
The other have Levi order $1$ if $|\phiup|{\geq}2$ and  $2$
if $\phiup\,{=}\,\{\alpha_{2q-1}\}$ with $2{\leq}q{\leq}{p-1}$.\par
Indeed, in the last case $\xiup_{\phiup}(\alpha_{2q-1})\,{=}\,1\geq\sup_{\alpha\in\Rad}|\xiup(\alpha)|$
and hence no root of the form $\alpha_{2q-{1}}{+}\alpha$, with $\alpha\,{\in}\,\Rad$, may
belong to $\Qc_{\,\,\phiup}^{-n}$.
\par\smallskip\noindent
$\boxed{\textsc{$\gs$ of type\,\,AIII, AIV}}$ We have $\gs\,{\simeq}\,\su(p,q)$ for integers $1{\leq}p{\leq}q$.\par
\vspace{3pt} 
We have, for an orthonormal basis $e_{1},\hdots,e_{p+q}$ of $\R^{p+q}$,
\begin{gather*} 
\begin{cases}
 \Rad\,{=}\,\{{\pm}(e_{i}{-}e_{j})\,{\mid}\,1{\leq}i{<}j{\leq}p{+}q\},\\
 \Bz\,{=}\,\{\alpha_{i}{=}e_{i}{-}e_{i+1}\,{\mid}\,1{\leq}i{\leq}p{+}q{-}1\},
\end{cases} \;\;
\begin{cases}
 \sigmaup(e_{i})={-}e_{p+q+1-i}, \,1{\leq}i{\leq}p\vee q{<}i{\leq}p{+}q, \\
 \sigmaup(e_{i})={-}e_{i}, \qquad p{<}i{\leq}q.
\end{cases} \\
\begin{cases}
 \epi(\alpha_{i})=\alpha_{p{+}q{-}i}, & 1{\leq}i{<}j{\leq}p\,\vee\, q{<}i{<}p{+}q,\\
 \epi(\alpha_{i})=\alpha_{i}, & p{<}i{\leq}q.
\end{cases}
\end{gather*}\par
The CR algebra $(\gs,\qt_{\phiup})$ is fundamental iff $\epi(\alpha_{i})\,{\notin}\,\phiup$ for all
$\alpha_{i}\,{\in}\,\phiup{\setminus}\Bz^{\sigmaup}_{\bullet}$ and finitely nondegenerate
iff $\phiup{\setminus}\Bz^{\sigmaup}_{\bullet}$ either contains at most one element, or,
for a sequence $1{\leq}i_{1}{<}\cdots{<}i_{h}{\leq}p$, equals either one of 
\begin{equation*}
\begin{aligned}
 & \{\alpha_{i_{2j-1}}\,{\mid}\,1{\leq}(2j{-}1){\leq}h\}\,{\cup}\, \{\epi(\alpha_{i_{2j}})\,{\mid}\,1{\leq}2j{\leq}h\},\\
& \{\epi(\alpha_{i_{2j-1}})\,{\mid}\,1{\leq}(2j{-}1){\leq}h\}\,{\cup} \, \{\alpha_{i_{2j}}\,{\mid}\,1{\leq}2j{\leq}h\},
\end{aligned}
\end{equation*}
and 
\begin{equation*} 
\begin{cases}
 |\phiup\cap\Bz^{\sigmaup}_{\bullet}|\,{\leq}\,2, & \text{if $\phiup\cap\{\alpha_{p},\alpha_{q+1}\}=\emptyset$},\\
  |\phiup\cap\Bz^{\sigmaup}_{\bullet}|\,{\leq}\,1, & \text{if $\phiup\cap\{\alpha_{p},\alpha_{q+1}\}\neq\emptyset$}.
\end{cases}
\end{equation*}
\par
The CR algebra $(\gs,\qt_{\phiup})$ has Levi order $2$ if and only if, moreover,
$\phiup\,{\subseteq}\,\Bz^{\sigmaup}_{\bullet}$. \par
Indeed, if $\emptyset\,{\neq}\,\phiup\,{\subseteq}\,\Bz^{\sigmaup}_{\bullet}$ and
$\alpha$ is the largest positive imaginary root, then $\alpha\,{\in}\,\Qc_{\;\phiup}\,{\cap}\,\sigmaup(\Qc^{-n}_{\;\phiup})$
and $\xiup_{\phiup}(\alpha)\,{\geq}\, \sup_{\alpha\in\Rad}|\xiup(\alpha)|$. 

 \par\medskip

\noindent
$\boxed{\textsc{$\gs$ of type\,\, BI,\,BII}}$ We have $\gs\,{\simeq}\,\so(p,2\ell{+}1{-}p)$, for $1{\leq}p{\leq}\ell$. 
\par\vspace{3pt}
When $p\,{=}\,\ell$, we get the split real form and 
all minimal parabolic CR algebras $(\gs,\qt_{\phiup})$ are totally real. \par
If $1{\leq}p{<}\ell$, 
$(\gs,\qt_{\phiup})$ is 
fundamental iff $\phiup\,{\subseteq}\,\Bz^{\sigmaup}_{\bullet}$ and 
finitely nondegenerate if, moreover, $|\phiup|\,{=}\,1$.  
\par Let us describe the root system, the basis of positive simple roots and the symmetry $\sigmaup$ by 
\begin{equation*} 
\begin{cases}
 \Rad\,{=}\,\{{\pm}e_{i}\,{\mid}\, 1{\leq}i{\leq}\ell\}\,{\cup}\,\{{\pm}e_{i}{\pm}e_{j}\,{\mid}\,1{\leq}i{<}j{\leq}\ell\},\\
 \Bz\,{=}\,\{\alpha_{i}{=}e_{i}{-}e_{i+1}\,{\mid}\,1{\leq}i{\leq}\ell{-}1\}\,{\cup}\,\{\alpha_{\ell}{=}e_{\ell}\},
\end{cases}\;\; 
\begin{cases}
 \sigmaup(e_{i})\,{=}\,e_{i}, & 1{\leq}i{\leq}p,\\
 \sigmaup(e_{i})\,{=}\,{-}e_{i}, & p{+}1{\leq}i{\leq}\ell.
\end{cases}
\end{equation*}
If $\phiup\,{=}\,\{\alpha_{p+1}\}$, then 
\begin{equation*} 
\begin{cases}
 \Qc^{n}_{\;\{\alpha_{p+1}\}}\cap\Rad^{\sigmaup}_{\bullet}\,{=}\,\{e_{p+1}\}\,{\cup}\,
 \{e_{p+1}{\pm}e_{i}\,{\mid}\, p{+}2{\leq}i{\leq}\ell\},\\
 \beta\,{=}\,({-}e_{1}{-}e_{p{+}1})=\sigmaup(e_{p+1}{-}e_{1})\,{\in}\,\sigmaup(\Qc^{r}_{\;\{\alpha_{p+1}\}})\cap
 \Qc^{-n}_{\;\{\alpha_{p+1}\}},\\
 e_{p+1}{+}\beta={-}e_{1}\,\in\,\Qc^{-n}_{\;\{\alpha_{p+1}\}}\cap\sigmaup(\Qc^{-n}_{\;\{\alpha_{p+1}\}}),\\
 (e_{p+1}{\pm}e_{i}){+}\beta={-}e_{1}{\pm}e_{i}\,\in\,\Qc^{-n}_{\;\{\alpha_{p+1}\}}\cap\sigmaup(\Qc^{-n}_{\;\{\alpha_{p+1}\}}),
 \;\forall p{+}2{\leq}i{\leq}\ell,
\end{cases}
\end{equation*}
shows that $(\gs,\qt_{\{\alpha_{p+1}\}})$ has Levi order~$1$.
\par\smallskip
If $\phiup\,{=}\,\{\alpha_{q}\}$ with $p{+}2\,{\leq}\,q\,{\leq}\,\ell$, then 
$e_{p+1}{+}e_{p+2}\,{\in}\,\Qc^{n}_{\;\phiup}{\cap}\Rad^{\sigmaup}_{\;\bullet}$
and $$2\,{=}\,\xiup_{\{\alpha_{q}\}}(e_{p+1}{+}e_{p+2})\,{\geq}\,{\sup}_{\beta\in\Rad}|\xiup_{\{\alpha_{q}\}}(\beta)|$$
implies that $(\gs,\qt_{\{\alpha_{q}\}})$ has Levi order~$2$.

\par \smallskip
\noindent
$\boxed{\textsc{$\gs$ of type\,\, CII}}$ We have $\gs\,{\simeq}\,\spt_{p,2\ell-p}$ for some integer $1{\leq}p{<}\ell$. With 
\begin{align*} 
&\begin{cases}
 \Rad\,{+}\,\{{\pm}2e_{i}\,{\mid}\,1{\leq}i{\leq}\ell\}\,{\cup}\,
 \{{\pm}e_{i}{\pm}e_{j}\,{\mid}\,1{\leq}i{<}j{\leq}\ell\},\\
 \Bz\,{=}\,\{\alpha_{i}{=}e_{i}{-}e_{{i+1}}\,{\mid}\,1{\leq}i{\leq}\ell{-}1\}\,{\cup}\,
 \{\alpha_{\ell}{=}2e_{\ell}\},
\end{cases} \\
&
\begin{cases}
 \sigmaup(e_{2i-1})\,{=}\,e_{2i},\; \sigmaup(e_{2i})\,{=}\,e_{2i-1}, & 1{\leq}i{\leq}p,\\
 \sigmaup(e_{i})\,{=}\,{-}e_{i}, & 2p{+}1{\leq}i{\leq}\ell,
\end{cases}
\end{align*}
we have
\begin{equation*}
 \Bz_{\bullet}^{\sigmaup}
 =\{\alpha_{2h-1}\mid 1{\leq}h{\leq}p\}\cup\{\alpha_{h}\mid 2p{+}1{\leq}h{\leq}\ell\}.
\end{equation*}
Fundamental $(\gs,\qt_{\phiup})$ have $\phiup\,{\subseteq}\,\Bz_{\bullet}^{\sigmaup}$; those which are
finitely nondegenerate have $|\{\alpha_{h}\,{\in}\,\phiup\,{\mid}\,2p{<}h{\leq}\ell\}|\,{\leq}\,1$. 
Hence, when $|\phiup|\,{\geq}\,2$ and $(\gs,\qt_{\phiup})$ is finitely nondegenerate, 
each root of $\phiup$ belongs to a different connected component of $\Rad^{\sigmaup}_{\;\bullet}$.
By Theorem\,\ref{t5.10} this implies that $(\gs,\qt_{\phiup})$ has Levi order $1$.\par
We are left  to compute Levi orders 
when $|\phiup|\,{=}\,1$. 
\par\noindent
1) \; Assume that $\phiup\,{=}\,\{\alpha_{2q-1}\}$, with $1{\leq}q{\leq}p$. \par
Then $\Qc^{n}_{\;\phiup}\,{\cap}\,\Rad_{\;\bullet}^{\sigmaup}\,{=}\,\{\alpha_{2q-{1}}\}$
and $\beta\,{=}\,{-}2e_{q-1}\,{=}\,\sigmaup({-}e_{2q})\,{\in}\,\sigmaup(\Qc^{r}_{\;\{\alpha_{2q-1}\}}\cap
\Qc_{\;\{\alpha_{2q-1}\}}^{-n})$ and
$\alpha_{2q-1}{+}\beta\,{=}\,{-}e_{2q-1}{-}e_{2q}
\,{\in}\,\Qc^{-n}_{\;\{\alpha_{2q-1}\}}\cap\sigmaup(\Qc^{-n}_{\;\{\alpha_{2q-1}\}})$
shows that also in this case the Levi order is $1$.
 \par\noindent
 2)\; Assume that $\phiup\,{=}\,\{\alpha_{q}\}$ with $2p{+}1{\leq}q{\leq}\ell$.
 Then $2e_{q}\,{\in}\,\Qc_{\;\{\alpha_{q}\}}^{n}\,{\cap}\,\Rad^{\sigmaup}_{\;\bullet}$
 and 
\begin{equation*}
 2=\xiup_{\{\alpha_{q}\}}(2e_{q})\geq{\sup}_{\beta\in\Rad}|\xiup_{\{\alpha_{q}\}}(\beta)|
\end{equation*}
shows that the Levi order of $(\gs,\qt_{\{\alpha_{q}\}})$ is $2$.

\par \medskip
\noindent
$\boxed{\textsc{$\gs$ of type\,\, DI, DII}}$ We have $\gs\,{\simeq}\,\so_{p,2\ell-p}$, with $1{\leq}p{\leq}\ell$.
\par \vspace{3pt}
For $p\,{=}\,\ell$ we obtain the split real form and $(\gs,\qt_{\phiup})$ is totally real for every $\phiup\,{\subseteq}\,
\Bz$. For $p\,{=}\,\ell{-}1$ we obtain the quasi-split real form. There are no imaginary roots and
$(\gs,\qt_{\phiup})$ is fundamental if and only if $\phiup$ equals either $\{\alpha_{\ell-1}\}$ or
$\{\alpha_{\ell}\}$. In both cases it is also finitely nondegenerate and of Levi order $1$.\par 
If $1{\leq}p{\leq}\ell{-}2,$ then $(\gs,\qt_{\phiup})$ is fundamental iff $\phiup\,{\subseteq}\,
\Bz^{\sigmaup}_{\bullet}\,{=}\,\{\alpha_{p+1},\hdots,\alpha_{\ell}\}$ and finitely nondegenerate 
in the following cases 
\begin{enumerate}
 \item $p\,{\geq}\,2$ and $|\phiup|=1$ or $\phiup\,{=}\,\{\alpha_{\ell-1},\alpha_{\ell}\}$;
 \item $p\,{=}\,1$ and either $\phiup\,{=}\,\{\alpha_{q}\}$ with $2{\leq}q{\leq}\ell{-}2$, or
 $\phiup\,{=}\,\{\alpha_{\ell-1},\alpha_{\ell}\}$.
\end{enumerate}

 \par
To discuss the different cases, we recall that, for an orthonormal basis $e_{1},\hdots,e_{\ell}$ of $\R^{\ell},$
we can set 
\begin{equation*}
\begin{cases}
 \Rad\,{=}\,\{{\pm}e_{i}{\pm}e_{j}\,{\mid}\,1\,{\leq}\,i\,{<}\,j\,{\leq}\,\ell\},\\
 \Bz\,{=}\,\{\alpha_{i}\,{=}\,e_{i}{-}e_{i+1}\,{\mid}\,1{\leq}i{\leq}i{-}1\}\cup\{\alpha_{\ell}\,{=}\,e_{\ell-1}{+}e_{\ell}\},
\end{cases}
 \end{equation*}
while $\sigmaup$ is the symmetry
\begin{equation*}
 \sigmaup(e_{i})= 
\begin{cases}
 \; e_{i}, & 1{\leq}i{\leq}p,\\
 {-}e_{i}, & p{<}i{\leq}\ell.
\end{cases}
\end{equation*}

\par
 We consider first the cases where $p{\geq}1$ and $\phiup\,{=}\,\phiup_{q}\,{=}\,\{\alpha_{q}\}$ with
$p{<}q{\leq}\ell{-}2$ or $\phiup\,{=}\,\phiup_{\ell-1}\,{=}\,\{e_{\ell-1},e_{\ell}\}$. Then  
$\Qc_{\;\phiup_{q}}\,{=}\,\{\alpha\,{\in}\,\Rad\,{\mid}\,\xiup_{q}(\alpha){\geq}0\}$ with 
\begin{equation*}
 \xiup_{q}(e_{i})=
\begin{cases}
 1, & 1{\leq}i{\leq}q,\\
 0, & q{<}i{\leq}\ell. 
\end{cases}
\end{equation*}
We have 
\begin{equation*} 
\begin{cases}
 \Qc^{n}_{\;\phiup_{p+1}}=\{e_{p+1}{\pm}e_{i}\mid p{+}2{\leq}i{\leq}\ell\},\\
 \sigmaup(e_{p+1}{-}e_{1})={-}e_{p+1}-e_{1}\,{\in}\,\sigmaup(\Qc^{r}_{\;\phiup})\cap\Qc^{-n}_{\;\phiup},\\
(e_{p+1}{\pm}e_{i}){+}({-}e_{p+1}{-}e_{1})={-}e_{1}{\pm}e_{i}\in\Qc^{-n}_{\;\phiup}\cap\sigmaup(\Qc^{-n}_{\;\phiup})
\end{cases}
\end{equation*}
and hence $(\gs,\qt_{\{\alpha_{p+1}\}})$ has Levi order $1$.\par
If $p{+}1\,{<}\,q\,{\leq}\,\ell{-}1$
$$\alpha\,{=}\,\alpha_{p+1}
{+}\alpha_{\ell-1}{+}\alpha_{\ell}{+}2{\sum}_{i=p+2}^{\ell{-}2}\alpha_{i}\,
{\in}\,\Qc^{n}_{\;\phiup}\,{\cap}\,\Rad^{\sigmaup}_{\;\bullet}$$
 and 
$\xiup_{\phiup}(\alpha)\,{\geq}\,{\sup}_{\beta\in\Rad}|\xiup_{\phiup}(\beta)|$ 
shows that $(\gs,\qt_{\phiup})$ has Levi order $2$.\par\smallskip
By the symmetry of the Satake diagram, the cases of $\phiup$ equal to $\{\alpha_{\ell-1}\}$ or
$\{\alpha_{\ell}\}$ have the same Levi order. Take $\phiup\,{=}\,\{\alpha_{\ell}\}$. Then 
\begin{equation*}
 \xiup_{\{\alpha_{\ell}\}}(e_{i})\,{=}\,\tfrac{1}{2},\;\;\forall 1{\leq}i{\leq}\ell
\end{equation*}
and $1\,{=}\,\xiup(\alpha_{\ell}){\geq}\,\sup_{\beta\in\Rad}|\xiup_{\{\alpha_{\ell}\}}(\beta)|$
shows that also in this case the Levi order is~$2$.

\par \medskip
\noindent
$\boxed{\textsc{$\gs$ of type\,\, DIIIa}}$ We have $\gs\,{\simeq}\,\su^{*}_{2p}(\Hb)$, for an integer $p{\geq}2$.
In all these cases it is finitely nondegenerate and has Levi order $1$ when $|\phiup|\,{\geq}2$.
\par
\vspace{3pt}\noindent
The conjugation on the root system is 
\begin{equation*}
\begin{cases}
 \sigmaup(e_{2i-1})=e_{2i}, \\ 
 \sigmaup(e_{2i})=e_{2i-1},  
\end{cases}1{\leq}i{\leq}p.
\end{equation*}
The CR algebra $(\gs,\qt_{\phiup})$ is fundamental if and only if $\Phi{\subseteq}\Bz^{\sigmaup}_{\;\bullet}
{=}\{\alpha_{2i-1}{\mid} 1{\leq}i{\leq}p\}$ and in all these cases is also finitely nondegnerate. 
By Theorem\,\ref{t5.10} it has Levi order $1$ when $|\phiup|\,{\geq}\,2$. 
\par
Let us consider the case where $\phiup\,{=}\,\{\alpha_{2q-1}\}$ for some $1{\leq}q{\leq}p.$ 
\par
We distinguish two cases.
\par\noindent
1) 
If $\phiup\,{=}\,\{\alpha_{2q-1}\}$ with $1{\leq}q{<}p,$ then the 
roots in $\sigmaup(\Qc_{\;\phiup}){\cap}\Qc_{\;\phiup}^{-n}$ which are not orthogonal to $\alpha_{2q-1}$
are either of the form ${-}e_{2q-1}{\pm}e_{h},$ with $h{>}2q$, or of the form $e_{2q}{-}e_{h}$ with $h{\leq}2q{-}2$.
We have 
\begin{equation*}\begin{cases}
 \alpha_{2q-1}{+}({-}e_{2q-1}{\pm}e_{h})=e_{2q}{\pm}e_{h}\notin \Qc^{-n}_{\;\phiup}\cap\sigmaup(\Qc_{\;\phiup}^{-n})
 & \text{for $h{>}2q$},\\
  \alpha_{2q-1}{+}(e_{2q}{-}e_{h})=e_{2q-1}{-}e_{h}\notin \Qc^{-n}_{\;\phiup}\cap\sigmaup(\Qc_{\;\phiup}^{-n})
 & \text{for $1{\leq}h{\leq}2q{-}2$}.
 \end{cases}
\end{equation*}
This shows that $k_{\phiup}^{\sigmaup}(\alpha_{2q-1})\,{\neq}\,1$ and therefore
$(\gs,\qt_{\phiup})$ has Levi order $2$. \par \noindent
2)
Consider finally the case $\phiup\,{=}\,\{\alpha_{2p-1}\}$. The roots 
 in $\sigmaup(\Qc_{\;\phiup}){\cap}\Qc_{\;\phiup}^{-n}$ which can be added to $\alpha_{2p-1}$
are of the form $e_{2p}{-}e_{h}$ with $h{\leq}2p{-}2$. Since 
\begin{equation*}
  \alpha_{2p-1}{+}(e_{2p}{-}e_{h})=e_{2p-1}{-}e_{h}\notin \Qc^{-n}_{\;\phiup}\cap\sigmaup(\Qc_{\;\phiup}^{-n})
 \;\;\text{for $1{\leq}h{\leq}2p{-}2$},
\end{equation*}
we have $k_{\phiup}^{\sigmaup}(\alpha_{2p-1})\,{\neq}\,1$ and therefore
$(\gs,\qt_{\phiup})$ has Levi order $2$.

\par \smallskip
\noindent
$\boxed{\textsc{$\gs$ of type\,\, DIIIb}}$ We have $\gs\,{\simeq}\,\su^{*}_{2p+1}(\Hb)$, for an integer $p{\geq}2$.
The conjugation on the root system is 
\begin{equation*}
\begin{cases}
 \sigmaup(e_{2i-1})=e_{2i}, & 1{\leq}i{\leq}p,\\ 
 \sigmaup(e_{2i})=e_{2i-1},  & 1{\leq}i{\leq}p,\\
 \sigmaup(e_{2p+1})={-}e_{2p+1}.
\end{cases}
\end{equation*}\par
The minimal parabolic fundamental CR algebras $(\gs,\qt_{\phiup})$
have\par
\centerline{$\phiup\,{\subseteq}\,\{\alpha_{2i-1}\,{\mid}\,1{\leq}i{\leq}p\}\,{\cup}\,\{\alpha_{2p},\alpha_{2p+1}\}$,\;
with\; $|\phiup\,{\cap}\,\{\alpha_{2p},\alpha_{2p+1}\}|\,{\leq}1$}
\par\noindent
and are all finitely nondegenerate. \par 
By repeating the previous arguments, it turns out that $(\gs,\qt_{\phiup})$ is 
\begin{itemize}
\item $1$-nondegenerate when
$\phiup$ contains $\alpha_{2p}$ or $\alpha_{2p+1}$, or $|\phiup|\,{\geq}\,2$;
\item $2$-nondegenerate when $\phiup\,{=}\,\{\alpha_{2q-1}\}$ for $1{\leq}q{\leq}p$.
\end{itemize}
\par\medskip
To treat in the following real forms of the exceptional Lie algebras of type $\textsc{E}$,
we shall use suitable root sytem, in which we will use elements of the following form.
After fixing 
an orthonormal basis $e_{1},\hdots,e_{8}$ of $\R^{8}$,
we set, for every permutation  $(i_{1},\hdots,i_{8})$ of $\{1,\hdots,8\}$,
\begin{equation} 
\begin{cases}
 \zetaup_{\emptyset}={-}\tfrac{1}{2}\left(e_{1}+\cdots+e_{8}\right),\\
 \zetaup_{i_{1}}=\tfrac{1}{2}\left(e_{i_{1}}{-}e_{i_{2}}{-}\cdots{-}e_{i_{8}}\right),\\
 \zetaup_{i_{1},\hdots,i_{h}}=\tfrac{1}{2}\left(e_{i_{1}}{+}\cdots{+}e_{i_{h}}{-}e_{i_{h+1}}{-}\cdots{-}e_{i_{8}}\right),
 & 2{\leq}h{\leq}6,\\
  \zetaup_{i_{1},\hdots,i_{7}}=\tfrac{1}{2}\left(e_{i_{1}}{+}\cdots{+}e_{i_{7}}{-}e_{i_{8}}\right),\\
 \zetaup_{i_{1},\hdots,i_{8}}=\tfrac{1}{2}\left(e_{1}+\cdots+e_{8}\right).
\end{cases}
\end{equation}

\par\smallskip\noindent
$\boxed{\textsc{$\gs$ of type\,\,EII}}$ This is a quasi-split real form, which, in  the representation 
\begin{equation} \label{e6.2}
\begin{cases}
 \Rad=\{{\pm}(e_{i}{-}e_{j}){\mid}{1{\leq}i{<}j{\leq}6}\}\,{\cup}\,\{{\pm}(e_{7}{-}e_{8})\}\,{\cup}\,
 \{{\pm}\zetaup_{i,j,k,7}{\mid}
 1{\leq}i{<}j{<}k{\leq}6\},\\
 \Bz=\{\alpha_{i}{=}e_{i}{-}e_{i+1}\,{\mid}\,1{\leq}i{\leq}5\}\cup\{\alpha_{6}=\zetaup_{4,5,6,7}\},
\end{cases}
\end{equation} 
is described by the symmetry
\begin{equation*}
 \epi(e_{i})=
\begin{cases}
{-}e_{7-i}, & 1{\leq}i{\leq}6,\\
{-}e_{9-i}, & i=7,8.
\end{cases}
\end{equation*}
The $\phiup$'s for which $(\gs,\qt_{\phiup})$ is fundamental are those with 
$\phiup\,{\cap}\,\epi(\phiup)\,{=}\,\emptyset$. 
Those for which it is also finitely nondegenerate are those with $|\phiup|\,{=}\,1$
and $\phiup\,\subset\,\{\alpha_{1},\alpha_{2},\alpha_{4},\alpha_{5}\}$
and also $\{\alpha_{1},\alpha_{4}\}$, $\{\alpha_{2},\alpha_{5}\}$.
In all these cases the Levi order is $1$, since $\Rad$ does not contain imaginary roots.
\par \smallskip
\noindent
$\boxed{\textsc{$\gs$ of type\,\, EIII}}$ This is a root system of type $\textsc{E}$, which, in
the setting \eqref{e6.2},
is defined by the symmetry
\begin{equation*} 
\begin{cases}
 \sigmaup(e_{1}){=}{-}e_{6},\; \sigmaup(e_{6}){=}{-}e_{1},\\
 \sigmaup(e_{7}){=}{-}e_{8},\; \sigmaup(e_{8}){=}{-}e_{7},\\
 \sigmaup(e_{i}){=}{-}e_{i}, \quad 2{\leq}i{\leq}5.
\end{cases}
\end{equation*}
The $\phiup$'s for which $(\gs,\qt_{\phiup})$ is fundamental are those 
 containing  neither $\alpha_{6}$ nor the couple $\{\alpha_{1},\alpha_{5}\}$ and 
they are all finitely nondegenerate, with the exceptions: 
\begin{equation*} \{\alpha_{1},\alpha_{2},\alpha_{3}\},\;\{\alpha_{3},\alpha_{4},\alpha_{5}\},\;
 \{\alpha_{1},\alpha_{2},\alpha_{3},\alpha_{4}\}\;\;\text{and}\;\;  \{\alpha_{2},\alpha_{3},\alpha_{4},\alpha_{5}\}.
\end{equation*}
Those containing either $\alpha_{1}$ or $\alpha_{5}$ are $1$-nondegenerate, because $\phiup$
contains a nonimaginary root. \par
Let us consider the cases where $\phiup\,{\subseteq}\,\Bz_{\bullet}^{\sigmaup}$.\par 
If either
$\phiup\,{=}\,\{\alpha_{2},\alpha_{3},\alpha_{4}\}\,{=}\,\Bz_{\bullet}^{\sigmaup}$, or $\phiup\,{=}\,\{\alpha_{2},\alpha_{4}\},$ 
 by Theorem\,\ref{t5.10} the single root in $\Qc^{n}_{\;\phiup}\,{\cap}\,\Rad^{\sigmaup}_{\;\phiup}$ which
 may have Levi order $2$ is $\alpha\,{=}\,\alpha_{2}{+}\alpha_{3}{+}\alpha_{4}\,{=}\,e_{2}{-}e_{5}$.\par
Since $(e_{6}{-}e_{2})\,{=}\,\sigmaup(e_{2}{-}e_{1})\,{\in}\,\sigmaup(\Qc^{r}_{\;\phiup})\,{\cap}\,\Qc_{\;\phiup}^{-n}$
 and 
\begin{equation*}
 (e_{2}{-}e_{5})+(e_{6}{-}e_{2})=e_{6}{-}e_{5}\in\Qc^{-n}_{\;\phiup}\cap\sigmaup(\Qc^{-n}_{\;\phiup}),
\end{equation*}
 we obtain $\ktt^{\sigmaup}_{\phiup}(\alpha)\,{=}\,1$, showing that $(\gs,\qt_{\{\alpha_{2},\alpha_{3},\alpha_{4}\}})$
 and $(\gs,\qt_{\{\alpha_{2},\alpha_{4}\}})$
 have Levi order $1$. 
 
\par
The cases with $\phiup\,{\subset}\,\Bz_{\bullet}^{\sigmaup}$ and $|\phiup|=2$ are
$1$-nondegenerate: the argument is the one used in the discussion of $\textsc{AIII}$.\par
Let $\phiup\,{=}\,\{\alpha_{2}\}.$ Then $\Qc_{\;\phiup}\,{=}\,\{\alpha\,{\in}\,\Rad\,{\mid}\,\xiup(\alpha)\geq{0}\}$ with 
\begin{equation*} 
\begin{cases}
 \xiup_{\phiup}(e_{1})=1,\; \xiup_{\phiup}(e_{2})=1,\;\xiup(e_{3})=0,\, \xiup(e_{4})=0,\\
 \xiup_{\phiup}(e_{5})=0,\; \xiup_{\phiup}(e_{6})=0,\;\xiup(e_{7})=2,\, \xiup(e_{8})=0.
\end{cases}
\end{equation*}
We have $\Qc_{\;\phiup}^{n}\,{\cap}\,\Rad_{\;\bullet}^{\sigmaup},{=}\,\{e_{2}{-}e_{i}\,{\mid}\,3{\leq}i{\leq}5\}$ and 
\begin{equation*} 
\begin{cases}
 (e_{2}{-}e_{3})+\zetaup_{3,4,5,8}=\zetaup_{2,4,5,8}\in\Qc_{\;\phiup}^{-n}\,{\cap}\,\sigmaup(\Qc_{\;\phiup}^{-n}),\;\;
 \text{with $\zetaup_{3,4,5,8}\,{\in}\,\sigmaup(\Qc_{\;\phiup})\,{\cap}\,\Qc_{\;\phiup}^{-n}$},\\
  (e_{2}{-}e_{4})+\zetaup_{3,4,5,8}=\zetaup_{2,3,5,8}\in\Qc_{\;\phiup}^{-n}\,{\cap}\,\sigmaup(\Qc_{\;\phiup}^{-n}),\;\;
 \text{with $\zetaup_{3,4,5,8}\,{\in}\,\sigmaup(\Qc_{\;\phiup})\,{\cap}\,\Qc_{\;\phiup}^{-n}$},\\
  (e_{2}{-}e_{5})+\zetaup_{3,4,5,8}=\zetaup_{2,3,4,8}\in\Qc_{\;\phiup}^{-n}\,{\cap}\,\sigmaup(\Qc_{\;\phiup}^{-n}),\;\;
 \text{with $\zetaup_{3,4,5,8}\,{\in}\,\sigmaup(\Qc_{\;\phiup})\,{\cap}\,\Qc_{\;\phiup}^{-n}$},
\end{cases}
\end{equation*}
shows that $(\gs,\qt_{\{\alpha_{2}\}})$  has Levi order $1$. Symmetrically, also 
$(\gs,\qt_{\{\alpha_{4}\}})$   has Levi order $1$. \par
Let us consider $\phiup\,{=}\,\{\alpha_{3}\}$. 
Then $\Qc_{\;\phiup}\,{=}\,\{\alpha\,{\in}\,\Rad\,{\mid}\,\xiup(\alpha)\geq{0}\}$ with 
\begin{equation*} 
\begin{cases}
 \xiup_{\phiup}(e_{1})=1,\; \xiup_{\phiup}(e_{2})=1,\;\xiup(e_{3})=1,\, \xiup(e_{4})=0,\\
 \xiup_{\phiup}(e_{5})=0,\; \xiup_{\phiup}(e_{6})=0,\;\xiup(e_{7})=3,\, \xiup(e_{8})=0.
\end{cases}
\end{equation*}
We have $\Qc_{\;\phiup}^{n}\,{\cap}\,\Rad_{\;\bullet}^{\sigmaup}\,{=}\,\{e_{2}{-}e_{i}\,{\mid}\,4{\leq}i{\leq}5\}\,{\cup}
\,\{e_{3}{-}e_{i}\,{\mid}\,4{\leq}i{\leq}5\}$ and
\begin{equation*} 
\begin{cases}
 (e_{2}{-}e_{4})+\zetaup_{1,4,5,8}=\zetaup_{1,2,5,8}\in\Qc_{\;\phiup}^{-n}\,{\cap}\,\sigmaup(\Qc_{\;\phiup}^{-n}),\;\;
 \text{with $\zetaup_{1,4,5,8}\,{\in}\,\sigmaup(\Qc_{\;\phiup})\,{\cap}\,\Qc_{\;\phiup}^{-n}$},\\
  (e_{2}{-}e_{5})+\zetaup_{1,4,5,8}=\zetaup_{1,2,4,8}\in\Qc_{\;\phiup}^{-n}\,{\cap}\,\sigmaup(\Qc_{\;\phiup}^{-n}),\;\;
 \text{with $\zetaup_{1,4,5,8}\,{\in}\,\sigmaup(\Qc_{\;\phiup})\,{\cap}\,\Qc_{\;\phiup}^{-n}$},
 \\
  (e_{3}{-}e_{4})+\zetaup_{1,4,5,8}=\zetaup_{1,3,5,8}\in\Qc_{\;\phiup}^{-n}\,{\cap}\,\sigmaup(\Qc_{\;\phiup}^{-n}),\;\;
 \text{with $\zetaup_{1,4,5,8}\,{\in}\,\sigmaup(\Qc_{\;\phiup})\,{\cap}\,\Qc_{\;\phiup}^{-n}$}\\
  (e_{3}{-}e_{5})+\zetaup_{1,4,5,8}=\zetaup_{1,3,4,8}\in\Qc_{\;\phiup}^{-n}\,{\cap}\,\sigmaup(\Qc_{\;\phiup}^{-n}),\;\;
 \text{with $\zetaup_{1,4,5,8}\,{\in}\,\sigmaup(\Qc_{\;\phiup})\,{\cap}\,\Qc_{\;\phiup}^{-n}$}.
\end{cases}
\end{equation*}
shows that also $(\gs,\qt_{\{\alpha_{3}\}})$  has Levi order $1$.
\par \medskip
\noindent
$\boxed{\textsc{$\gs$ of type\,\, \textsc{E\,IV}}}$ To discuss the real form \textsc{E\,IV},
we found convenient to describe 
the root system  and a basis of simple roots  by 
(the upper sign is for the root in $\Rad^{+}$)
\begin{equation} 
\begin{cases}
 \Rad=\{{\pm}(e_{i}{\pm}e_{j}),\,{\pm}\zetaup_{i,j}{\mid}{1{\leq}i{<}j{\leq}5}\}\,{\cup}\,\{{\pm}\zetaup_{\emptyset}\}\,{\cup}\,
 \{{\mp}\zetaup_{i,6,7,8}{\mid}
 1{\leq}i{\leq}5\},\\
 \Bz=\{\alpha_{i}{=}e_{i}{-}e_{i+1}\,{\mid}\,1{\leq}i{\leq}4\}\cup\{\alpha_{5}{=}e_{4}{+}e_{5}\}\cup
 \{\alpha_{6}=\zetaup_{\emptyset}\}.
\end{cases}
\end{equation}
The symmetry  yielding the form $\textsc{EIV}$ is defined by 
\begin{equation*} 
\begin{cases}
 \sigmaup(e_{i}){=}e_{i},\; i=1,6,7,8,\\
 \sigmaup(e_{i}){=}{-}e_{i},\; 1=2,3,4,5.
\end{cases}
\end{equation*}
The corresponding Satake diagram is \par\medskip
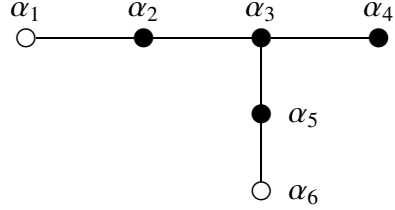
\begin{equation*}
\xymatrix @M=0pt @R=4pt @!C=20pt{
  &\alpha_{1}&\alpha_{2}&\alpha_{3}&\alpha_{4}\\
  &\medcirc\ar@{-}[r]
  &\medbullet\ar@{-}[r]
  &\medbullet\ar@{-}[r]\ar@{-}[dddd]
  &\medbullet \\
  \\
  \\
  \\
  &&&\medbullet\ar@{-}[dddd] &\!\!\!\!\!\!\!\!\!\!\!\!\!\!\!\!\!\!\!\!\!\!\!\!\!\!\!\!\alpha_{5}
  \\
  \\
  \\
  \\
 & &&\medcirc&\!\!\!\!\!\!\!\!\!\!\!\!\!\!\!\!\!\!\!\!\!\!\!\!\!\!\!\!\alpha_{6}
  }
\end{equation*}
 
\par\bigskip 
The fundamental $(\gs,\qt_{\phiup})$'s have 
$\phiup\,{\subseteq}\,\Bz_{\bullet}^{\sigmaup}$ and they are finitely nondegenerate iff $|\phiup|{\leq}2$
and $\phiup\,{\neq}\,\{\alpha_{3},\alpha_{4}\}$.
\par
By the argument used in case $\textsc{DII}$, $(\gs,\qt_{\phiup})$  has Levi order $1$ when 
$\phiup$ equals either $\{\alpha_{2}\}$ or $\{\alpha_{5}\}$.  \par
Let $\phiup\,{=}\,\{\alpha_{3}\}$. We have 
\begin{equation*}
 \xiup_{\{\alpha_{3}\}}(e_{i})=
\begin{cases}
 \;1, & i=1,2,3,\\
 \;0, & i=4,5,\\
 {-}1, & i=6,7,8.
\end{cases}
\end{equation*}
The root $\beta\,{=}\,(e_{2}{+}e_{3})$ belongs to $\Qc^{n}_{\;\phiup}\,{\cap}\,\Rad_{\;\bullet}^{\sigmaup}$ 
and $\xiup_{\{\alpha_{3}\}}(\beta)\,{=}2.$ Since only the roots $\zetaup_{4,6,7,8}$ and
$\zetaup_{5,6,7,8}$ have a $\xiup_{\{\alpha_{3}\}}$ less than $({-}2)$, and they both belong to
$\Qc_{\;}^{-n}\,{\cap}\,\sigmaup(\Qc^{-n}_{\phiup})$, we obtain $k_{\{\alpha_{3}\}}^{\sigmaup}(\beta)\,{=}\,2$.
Hence $(\gs,\qt_{\{\alpha_{3}\}})$ has Levi order $2$. \par\smallskip
Let $\phiup\,{=}\,\{\alpha_{4}\}$. We have 
\begin{equation*}
 \xiup_{\{\alpha_{4}\}}(e_{i})=
\begin{cases}
 \;\;\tfrac{1}{2}, & i=1,2,3,4\\
 {-}\tfrac{1}{2}, & i=5,6,7,8
\end{cases}
\end{equation*}
and $\Qc^{n}_{\;\phiup}\,{\cap}\,\Rad_{\;\bullet}^{\sigmaup}\,{=}\,\{e_{i}{-}e_{5}\,{\mid}\,2{\leq}i{\leq}4\}
\,{\cup}\,\{e_{i}{+}e_{j}\,{\mid}\,2{\leq}i{<}j{\leq}4\}$. Taking into account that 
$({-}e_{i}{-}e_{j})$, for $2{\leq}i{<}j{\leq}4$, and
$\zetaup_{1,6,7,8}\,{=}\,\sigmaup({-}\zetaup_{\emptyset})$ belong to $\sigmaup(\Qc^{r}_{\;\phiup})\,{\cap}\,
\Qc^{-n}_{\;\phiup}$, the equalities 
\begin{equation*}\begin{cases}
 (e_{i}{-}e_{5}){+}({-}e_{1}{-}\e_{i})=({-}e_{1}{-}\e_{5})\in\Qc^{-n}_{\;\phiup}\cap\sigmaup(\Qc^{-n}_{\;\phiup},
 & 2{\leq}i{\leq}4,\\
 (e_{i}{+}e_{j}){+}\zetaup_{1,6,7,8}=\zetaup_{1,i,j,6,7,8}\in\Qc^{-n}_{\;\phiup}\cap\sigmaup(\Qc^{-n}_{\;\phiup},
 & 2{\leq}i{<}j{\leq}4,\end{cases}
\end{equation*}
show that $(\gs,\qt_{\{\alpha_{4}\}})$  has Levi order $1$. 
\par\smallskip
Next we consider the subsets $\phiup$ with $|\phiup|\,{=}\,2$. \par
Let $\phiup\,{=}\,\phiup_{2,3}\,{=}\,\{\alpha_{2},\alpha_{3}\}$. 
We have $\Qc_{\;\phiup}\,{=}\,\{\alpha\,{\in}\,\Rad\,{\mid}\,\xiup_{\phiup_{2,3}}(\alpha){\geq}0\}$ with 
\begin{equation*}
 \xiup_{\phiup_{2,3}}= 
\begin{cases}
 \; 2, & i=1,2,\\
 \; 1, & i=3,\\
\; 0, & i=4,5,\\
 {-}\tfrac{5}{3}, & i=6,7,8.
\end{cases}
\end{equation*} 
Then 
\begin{equation*}
 \Qc^{n}_{\;\phiup_{2,3}}\cap\Rad^{\sigmaup}_{\;\phiup_{2,3}}=\{e_{i}{-}e_{j}\,{\mid}\, i=2,3,\; i{<}j{\leq}5\}
 \cup \{e_{i}{+}e_{j}\,{\mid} 2{\leq}{i}{\leq}3,\; i{\leq}j{\leq}5\}.
\end{equation*}
We note that $({-}e_{1}{-}e_{2})\,{=}\,\sigmaup(e_{2}{-}e_{1})$, ${-}\zeta_{2,3}\,{=}\,\sigmaup(\zetaup_{4,5})$
and $\zetaup_{1,6,7,8}{=}\,\sigmaup({-}\zetaup_{\emptyset})$ belong to $\sigmaup(\Qc^{r}_{\phiup_{2,3}})
\,{\cap}\,\Qc^{-n}_{\phiup_{2,3}}.$ Therefore 
\begin{equation*} 
\begin{cases}
 (e_{2}{-}e_{i})+({-}e_{1}{-}e_{2})={-}e_{1}{-}e_{i}\in\Qc_{\;\phiup_{2,3}}^{-n}\cap\sigmaup(\Qc_{\;\phiup_{2,3}}^{-n}),
 & i=3,4,5,\\
 (e_{3}{-}e_{i})+({-}\zetaup_{2,3})={-}\zetaup_{2,i}\in\Qc_{\;\phiup_{2,3}}^{-n}\cap\sigmaup(\Qc_{\;\phiup_{2,3}}^{-n})
 , & i=4,5,\\
 (e_{2}{+}e_{i})+\zetaup_{1,6,7,8}=\zetaup_{1,2,i,6,7,8}\in\Qc_{\;\phiup_{2,3}}^{-n}\cap\sigmaup(\Qc_{\;\phiup_{2,3}}^{-n}), 
 & i=3,4,5,\\
  (e_{3}{+}e_{i})+\zetaup_{1,6,7,8}=\zetaup_{1,3,i,6,7,8}\in\Qc_{\;\phiup_{2,3}}^{-n}\cap\sigmaup(\Qc_{\;\phiup_{2,3}}^{-n}), 
 & i=4,5,
\end{cases}
\end{equation*}
shows that $(\gs,\qt_{\phiup_{2,3}})$  has Levi order $1$. By the symmetry or the Satake diagram $\textsc{E\,IV}$,
also $(\gs,\qt_{\phiup_{3,5}})$, with $\phiup_{3,5}\,{=}\,\{\alpha_{3},\alpha_{5}\}$  has Levi order $1$.\par\smallskip
Let now $\phiup\,{=}\,\phiup_{2,4}\,{=}\,\{\alpha_{2},\alpha_{4}\}$. 
In this case we take 
\begin{equation*}
 \xiup_{\phiup_{2,4}}= 
\begin{cases}
 \; \tfrac{3}{2}, & i=1,2,\\
 \; \tfrac{1}{2}, & i=3,4,\\
-\tfrac{1}{2}, & i=5,\\
 {-}\tfrac{7}{6}, & i=6,7,8.
\end{cases}
\end{equation*} 
By Theorem\,\ref{t5.10} the elements of $\Qc_{\;\phiup}\cap\sigmaup(\Qc^{-n}_{\;\phiup})$ 
which can have Levi order $2$ are roots of $\Qc^{n}_{\;\phiup}\,{\cap}\,\Rad^{\sigmaup}_{\;\bullet}$
whose support contains $\{\alpha_{2},\alpha_{4}\}$. Thus we need to check the Levi order of the
three roots $e_{2}{+}e_{3},\, e_{2}{+}e_{4},\; e_{2}{-}e_{5}$. We have $\sigmaup(e_{2}{-}e_{1})\,{=}\,
({-}e_{1}{-}e_{2})\,{\in}\,\sigmaup(\Qc^{r}_{\;\phiup})\,{\cap}\,\Qc^{-n}_{\;\phiup}$ and then 
\begin{equation*} 
\begin{cases}
 (e_{2}{+}e_{3}){+}({-}e_{1}{-}e_{2})= e_{3}{-}e_{1}\in\Qc^{-n}_{\;\phiup}\cap\sigmaup(\Qc^{-n}_{\;\phiup}),\\
  (e_{2}{+}e_{4}){+}({-}e_{1}{-}e_{2})= e_{4}{-}e_{1}\in\Qc^{-n}_{\;\phiup}\cap\sigmaup(\Qc^{-n}_{\;\phiup}),\\
   (e_{2}{-}e_{5}){+}({-}e_{1}{-}e_{2})={-}e_{1}{-}e_{5}\in\Qc^{-n}_{\;\phiup}\cap\sigmaup(\Qc^{-n}_{\;\phiup}),\\
\end{cases}
\end{equation*}
shows that $(\gs,\qt_{\{\alpha_{2},\alpha_{4}\}})$ has Levi order $1$.
By the symmetry of the Satake diagram $\textsc{E\,IV}$,
also $(\gs,\qt_{\phiup_{4,5}})$, with $\phiup_{4,5}\,{=}\,\{\alpha_{4},\alpha_{5}\}$  has Levi order $1$.\par\smallskip
Let now $\phiup\,{=}\,\phiup_{2,5}\,{=}\,\{\alpha_{2},\alpha_{5}\}.$ We have 
\begin{equation*}
 \xiup_{\phiup_{2,5}}(e_{i})= 
\begin{cases}
 \;\,2, & 1=1,2,\\
 \;\,1,&  i=3,4,5,\\
 {-}\tfrac{7}{3}, & i=6,7,8,
\end{cases}
\end{equation*}
By Theorem\,\ref{t5.10} the elements of $\Qc_{\;\phiup}\cap\sigmaup(\Qc^{-n}_{\;\phiup})$ 
which can have Levi order $2$ are roots of $\Qc^{n}_{\;\phiup}\,{\cap}\,\Rad^{\sigmaup}_{\;\bullet}$
whose support contains $\{\alpha_{2},\alpha_{5}\}$. Thus we need to check the Levi order of the
three roots $e_{2}{+}e_{3},\, e_{2}{+}e_{4},\; e_{2}{+}e_{5}$. We have $\sigmaup(e_{2}{-}e_{1})\,{=}\,
({-}e_{1}{-}e_{2})\,{\in}\,\sigmaup(\Qc^{r}_{\;\phiup})\,{\cap}\,\Qc^{-n}_{\;\phiup}$ and then 
\begin{equation*} 
\begin{cases}
 (e_{2}{+}e_{3}){+}({-}e_{1}{-}e_{2})= e_{3}{-}e_{1}\in\Qc^{-n}_{\;\phiup}\cap\sigmaup(\Qc^{-n}_{\;\phiup}),\\
  (e_{2}{+}e_{4}){+}({-}e_{1}{-}e_{2})= e_{4}{-}e_{1}\in\Qc^{-n}_{\;\phiup}\cap\sigmaup(\Qc^{-n}_{\;\phiup}),\\
   (e_{2}{+}e_{5}){+}({-}e_{1}{-}e_{2})=e_{5}{-}e_{1}\in\Qc^{-n}_{\;\phiup}\cap\sigmaup(\Qc^{-n}_{\;\phiup}),\\
\end{cases}
\end{equation*}
shows that $(\gs,\qt_{\{\alpha_{2},\alpha_{5}\}})$ has Levi order $1$.

\par \medskip
\noindent
$\boxed{\textsc{$\gs$ of type\,\, EVI}}$ This is a real form of $\textsc{E}_{7}$. Let us take
the root system of type $\textsc{E}_{7}$ and the basis of simple roots consisting of  
(the upper sign is for the root in $\Rad^{+}$)
\begin{equation} 
\begin{cases}
 \Rad=\{{\pm}(e_{i}{\pm}e_{j}),{\pm}\zetaup_{i,j},{\mp}\zetaup_{i,j,7,8}{\mid}{1{\leq}i{<}j{\leq}6}\}
 {\cup}\{{\pm}(e_{7}{+}e_{8}),{\pm}\zetaup_{\emptyset},{\mp}\zetaup_{7,8}\}\\
 \Bz=\{\alpha_{i}{=}e_{i}{-}e_{i+1}\,{\mid}\,1{\leq}i{\leq}5\}\cup\{\alpha_{6}{=}e_{5}{+}e_{6}\}\cup
 \{\alpha_{7}=\zetaup_{\emptyset}\}.
\end{cases}
\end{equation}
The symmetry  yielding the form $\textsc{EVI}$ is defined by 
\begin{equation*} 
\begin{cases}
 \sigmaup(e_{2i-1}){=}e_{2i},\; 1{\leq}i{\leq}4,\\
 \sigmaup(e_{2i}){=}e_{2i-1},\; 1{\leq}i{\leq}4.
\end{cases}
\end{equation*}
The corresponding Satake diagram is \par\medskip
 
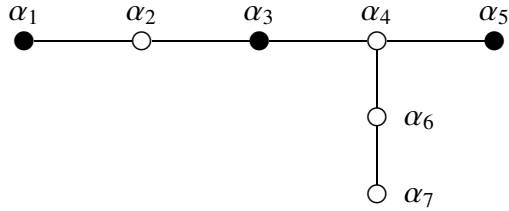
\begin{equation*}
  \xymatrix @M=0pt @R=4pt @!C=20pt{
 \alpha_{1} &\alpha_{2}&\alpha_{3}&\alpha_{4}&\alpha_{5}\\
  \medbullet\ar@{-}[r]&\medcirc\ar@{-}[r]
  &\medbullet\ar@{-}[r]
  &\medcirc\ar@{-}[r]\ar@{-}[dddd]
  &\medbullet \\
  \\
  \\
  \\
  &&&\medcirc\ar@{-}[dddd] &\!\!\!\!\!\!\!\!\!\!\!\!\!\!\!\!\!\!\!\!\!\!\!\!\!\!\!\!\alpha_{6}
  \\
  \\
  \\
  \\
 & &&\medcirc&\!\!\!\!\!\!\!\!\!\!\!\!\!\!\!\!\!\!\!\!\!\!\!\!\!\!\!\!\alpha_{7}
  }
\end{equation*}
\par\medskip
The $\phiup$'s yielding fundamental minimal parabolic CR algebras $(\gs,\qt_{\phiup})$
are those with $\phiup\,{\subseteq}\,\Bz_{\bullet}^{\sigmaup}\,{=}\,\{\alpha_{1},\alpha_{3},\alpha_{5}\}$
and they are all finitely nondegenerate. By the discussion of the case $\textsc{AII}$ we know
that those with $|\phiup|=2,3$ are $1$-nondegenerate. Let us discuss the cases where $|\phiup|\,{=}\,1$.\par
If $\phiup\,{=}\,\phiup_{1}\,{=}\{\alpha_{1}\}$, we have $\Qc_{\;\phiup_{1}}\,{=}\,\{\alpha\,{\in}\,\Rad\,{\mid}\,
\xiup_{1}(\alpha)\,{\geq}\,0\}$ with 
\begin{equation*} \xiup_{1}(e_{i})=
\begin{cases}
\;1, & i=1,\\
\;0, & 2{\leq}i{\leq}6,\\
{-}\tfrac{1}{2}, & i=7,8.
\end{cases}
\end{equation*}
Since $|\xiup_{1}(\alpha)|{\leq}1$ for all roots $\alpha$ of $\Rad$ and $\xiup_{1}(\alpha_{1})\,{=}1$,
we get $k_{\;\phiup_{1}}^{\sigmaup}(\alpha_{1})\,{=}\,2$ and hence $(\gs,\qt_{\;\phiup_{1}})$ has Levi order $2$.
\par
\smallskip
If $\phiup\,{=}\,\phiup_{3}\,{=}\{\alpha_{3}\}$, we have $\Qc_{\;\phiup_{3}}\,{=}\,\{\alpha\,{\in}\,\Rad\,{\mid}\,
\xiup_{3}(\alpha)\,{\geq}\,0\}$ with 
\begin{equation*} \xiup_{3}(e_{i})=
\begin{cases}
\;1, & i=1,2,3\\
\;0, & i=4,5,6,\\
{-}\tfrac{3}{2}, & i=7,8.
\end{cases}
\end{equation*}
By the discussion of $\textsc{DIIIa}$ we know that there is no root $\beta$ of the form ${\pm}e_{i}{\pm}e_{j}$
such that $\alpha_{3}{+}\beta\,{\in}\,\Qc^{-n}_{\;\phiup_{3}}\,{\cap}\,\sigmaup(\Qc^{-n}_{\;\phiup_{3}})$.
The negative $\zetaup$-roots in $\Qc_{\;\phiup_{3}}^{r}$ are $({-}\zetaup_{\emptyset})$ and 
$({-}\zetaup_{i,j})$ with $4{\leq}i{<}j{\leq}6$. We have $\sigmaup({-}\zetaup_{\emptyset})\,{=}\,({-}\zetaup_{\emptyset})$,
$\sigmaup({-}\zetaup_{4,5})\,{=}\,({-}\,\zetaup_{3,6})$, 
$\sigmaup({-}\zetaup_{4,6})\,{=}\,({-}\,\zetaup_{3,5})$,
$\sigmaup({-}\zetaup_{5,6})\,{=}\,({-}\,\zetaup_{5,6})$. Thus the $\zetaup$-roots in 
$\sigmaup(\Qc^{r}_{\;\phiup_{3}})\,{\cap}\,\Qc^{-n}_{\;\phiup_{3}}$ are $({-}\zetaup_{3,5})$ and $({-}\zetaup_{{3,6}})$
and both $\alpha_{3}{+}({-}\zetaup_{3,5})\,{=}\,({-}\zetaup_{4,5})$ and $\alpha_{3}{+}({-}\zetaup_{3,6})\,{=}\,({-}\zetaup_{4,6})$
are in $\Qc^{r}_{\;\phiup_{3}}$. This shows that $k_{\;\phiup_{3}}^{\sigmaup}(\alpha_{3})\,{=}\,2$ and hence
$(\gs,\qt_{\;\phiup_{3}})$ has Levi order $2$. 
\par
\smallskip
\par
\smallskip
If $\phiup\,{=}\,\phiup_{5}\,{=}\{\alpha_{5}\}$, we have $\Qc_{\;\phiup_{5}}\,{=}\,\{\alpha\,{\in}\,\Rad\,{\mid}\,
\xiup_{5}(\alpha)\,{\geq}\,0\}$ with 
\begin{equation*} \xiup_{5}(e_{i})=
\begin{cases}
\;\tfrac{1}{2}, & i=1,2,3,4,5\\
{-}\tfrac{1}{2}, & i=6,\\
{-}1, & i=7,8.
\end{cases}
\end{equation*}
By the discussion of $\textsc{DIIIa}$ we know that there is no root $\beta$ of the form ${\pm}e_{i}{\pm}e_{j}$
with $\alpha_{3}{+}\beta\,{\in}\,\Qc^{-n}_{\;\phiup_{3}}\,{\cap}\,\sigmaup(\Qc^{-n}_{\;\phiup_{3}})$.
The negative $\zetaup$-roots in $\Qc_{\;\phiup_{3}}^{r}$ are $({-}\zetaup_{\emptyset})$ and 
$({-}\zetaup_{i,6})$ with $1{\leq}i{\leq}5$. \par
We have $\sigmaup({-}\zetaup_{\emptyset})\,{=}\,({-}\zetaup_{\emptyset})$,
$\sigmaup({-}\zetaup_{2i-1,6})\,{=}\,({-}\,\zetaup_{2i,5})$,  
$\sigmaup({-}\zetaup_{2i,6})\,{=}\,({-}\,\zetaup_{2i-1,5})$, for $i\,{=}\,1,2$ and
$\sigmaup({-}\zetaup_{5,6})\,{=}\,({-}\,\zetaup_{5,6})$. 
Thus the $\zetaup$-roots in 
$\sigmaup(\Qc^{r}_{\;\phiup_{3}})\,{\cap}\,\Qc^{-n}_{\;\phiup_{3}}$ are $({-}\zetaup_{i,5})$ 
for $i\,{=}\,1,2,3,4$. We have $\alpha_{5}{+}({-}\zetaup_{i,5})\,{=}\,({-}\zetaup_{i,6})\,{\in}\,\Qc^{r}_{\;\phiup_{5}}$
for all $i\,{=}\,1,2,3,4$ and hence $k_{\;\phiup_{5}}^{\sigmaup}(\alpha_{5})=2$, showing that
$(\gs,\qt_{\;\phiup_{5}})$ has Levi order $2$.
\par \medskip
\noindent
$\boxed{\textsc{$\gs$ of type\,\, EVII}}$ This is a real form of $\textsc{E}_{7}$.
We keep the notation introduced above for $\textsc{EVI}$.
\par 
The symmetry  yielding $\textsc{EVII}$ is defined by 
\begin{equation*} \sigmaup(e_{i})=
\begin{cases}
\;\, e_{i},\; i=1,2,7,8,\\
\;\,{-}e_{i},\; i=3,4,5,6.
\end{cases}
\end{equation*}
The corresponding Satake diagram is \par\medskip
 
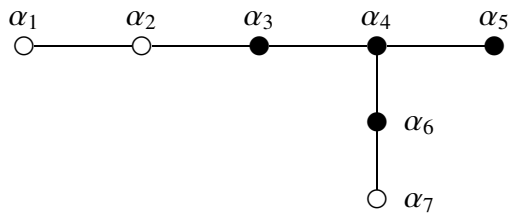
\begin{equation*}
  \xymatrix @M=0pt @R=4pt @!C=20pt{
 \alpha_{1} &\alpha_{2}&\alpha_{3}&\alpha_{4}&\alpha_{5}\\
  \medcirc\ar@{-}[r]&\medcirc\ar@{-}[r]
  &\medbullet\ar@{-}[r]
  &\medbullet\ar@{-}[r]\ar@{-}[dddd]
  &\medbullet \\
  \\
  \\
  \\
  &&&\medbullet\ar@{-}[dddd] &\!\!\!\!\!\!\!\!\!\!\!\!\!\!\!\!\!\!\!\!\!\!\!\!\!\!\!\!\alpha_{6}
  \\
  \\
  \\
  \\
 & &&\medcirc&\!\!\!\!\!\!\!\!\!\!\!\!\!\!\!\!\!\!\!\!\!\!\!\!\!\!\!\!\alpha_{7}
  }
\end{equation*}
\par\medskip
The fundamental minimal parabolic $(\gs,\qt_{\phiup})$ correspond to
$$\phiup\,{\subseteq}\,\Bz_{\bullet}^{\sigmaup}\,{=}\,\{\alpha_{3},\alpha_{4},\alpha_{5},\alpha_{6}\}$$ 
and 
are finitely nondegenerate if $|\phiup|{\leq}2$ and $\phiup\,{\neq}\,\{\alpha_{4},\alpha_{5}\}$.
By the discussion of $\textsc{E\,IV}$, we know that in the cases listed above 
$(\gs,\qt_{\phiup})$  has Levi order $1$ for $\phiup\,{\neq}\,\{\alpha_{4}\}$. \par
Let $\phiup\,{=}\,\phiup_{4}\,{=}\,\{\alpha_{4}\}$. Then $\Qc_{\phiup_{4}}\,{=}\,\{\alpha\,{\in}\,\Rad\,{\mid}\,
\xiup(\alpha){\geq}0\}$ with 
\begin{equation*}
 \xiup(e_{i})= 
\begin{cases}
 \;\,1, & i=1,2,3,4,\\
 \;\,0, & i=5,6\\
{-}2, & i=7,8.
\end{cases}
\end{equation*}
We have 
\begin{equation*}
 \Qc_{\phiup_{4}}^{n}\cap\Rad^{\sigmaup}_{\;\bullet}=\{e_{3}{+}e_{4},e_{3}{\pm}e_{5},e_{3}{\pm}e_{6},e_{4}{\pm}e_{5},
 e_{4}{\pm}e_{6}\}.
\end{equation*}
Since $\xiup(e_{3}{+}e_{4})\,{=}\,2$, there is no root $\beta$ of the form ${\pm}e_{i}{\pm}e_{j}$ in $\Rad$
for which $(e_{3}{+}e_{4}){+}\beta\,{\in}\,\Qc^{-n}_{\;\phiup_{4}}$. The $\zetaup$-roots
in $\Qc^{r}_{\phiup_{4}}$ are ${\pm}\zetaup_{\emptyset}$ and ${\pm}\zetaup_{5,6}$.
We have $\sigmaup({-}\zetaup_{\emptyset})=\zetaup_{1,2,7,8}$ and $\sigmaup({-}\zetaup_{5,6})\,{=}\,{-}\,\zetaup_{3,4}$.
Since 
\begin{equation*}
 (e_{3}{+}e_{4})+\zetaup_{1,2,7,8}={-}\zetaup_{5,6}\,{\in}\,\Qc_{\;\phiup_{4}}^{r}\;\;\text{and}\;\;
 (e_{3}{+}e_{4})+({-}\zetaup_{3,4})={-}\zetaup_{\emptyset}\,{\in}\,\Qc_{\;\phiup_{4}}^{r}, 
\end{equation*}
we obtain that $k_{\phiup_{4}}^{\sigmaup}(e_{3}{+}e_{4})\,{=}\,2$ and therefore
$(\gs,\qt_{\phiup_{4}})$ has Levi order $2$.
\par \medskip
\noindent
$\boxed{\textsc{$\gs$ of type\,\, E\,{IX}}}$ This is a real form of $\textsc{E}_{8}$.
Let us take
the root system of type $\textsc{E}_{8}$ and the basis of simple roots consisting of  
\begin{equation} 
\begin{cases}
 \Rad=\{{\pm}\zetaup_{\emptyset}\}\cup
 \{{\pm}(e_{i}{\pm}e_{j}),{\pm}\zetaup_{i,j},{\mid}{1{\leq}i{<}j{\leq}8}\}
 {\cup}\{\zetaup_{i,j,h,\ell}{\mid} 1{\leq}i{<}j{<}h{<}\ell{\leq}8\}\\
 \Bz=\{\alpha_{i}{=}e_{i}{-}e_{i+1}\,{\mid}\,1{\leq}i{\leq}6\}\cup\{\alpha_{7}{=}e_{6}{+}e_{7}\}\cup
 \{\alpha_{8}=\zetaup_{\emptyset}\}.
\end{cases}
\end{equation}
\par 
The symmetry  yielding $\textsc{E\,{IX}}$ is defined by 
\begin{equation*} \sigmaup(e_{i})=
\begin{cases}
\;\, e_{i},\; i=1,2,3,8,\\
\;\,{-}e_{i},\; i=4,5,6,7.
\end{cases}
\end{equation*}
The corresponding Satake diagram is \par\medskip
 
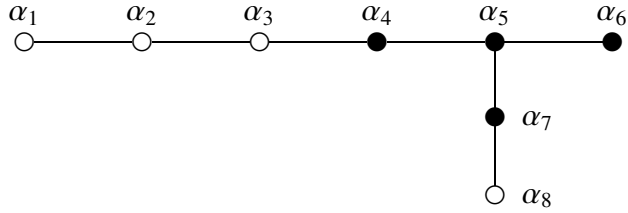
\begin{equation*}
  \xymatrix @M=0pt @R=4pt @!C=20pt{
 \alpha_{1} &\alpha_{2}&\alpha_{3}&\alpha_{4}&\alpha_{5}&\alpha_{6}\\
  \medcirc\ar@{-}[r]&\medcirc\ar@{-}[r]
  &\medcirc\ar@{-}[r]
  &\medbullet\ar@{-}[r]
  &\medbullet\ar@{-}[r]\ar@{-}[dddd]
  &\medbullet \\
  \\
  \\
  \\
  &&&&\medbullet\ar@{-}[dddd] &\!\!\!\!\!\!\!\!\!\!\!\!\!\!\!\!\!\!\!\!\!\!\!\!\!\!\!\!\alpha_{7}
  \\
  \\
  \\
  \\
 && &&\medcirc&\!\!\!\!\!\!\!\!\!\!\!\!\!\!\!\!\!\!\!\!\!\!\!\!\!\!\!\!\alpha_{8}
  }
\end{equation*}
\par\medskip
The fundamental minimal parabolic $(\gs,\qt_{\phiup})$ correspond to
$$\phiup\,{\subseteq}\,\Bz_{\bullet}^{\sigmaup}\,{=}\,\{\alpha_{4},\alpha_{5},\alpha_{6},\alpha_{7}\}$$ 
and 
are finitely nondegenerate if $|\phiup|{\leq}2$ and $\phiup\,{\neq}\,\{\alpha_{5},\alpha_{6}\}$.
By the discussion of $\textsc{EIV}$, we know that in the cases listed above 
$(\gs,\qt_{\phiup})$ has Levi order $1$ for $\phiup\,{\neq}\,\{\alpha_{5}\}$,
while the same argument employed in that case shows that $(\gs,\qt_{\{\alpha_{5}\}})$
has Levi order $2$.
\par \medskip
\noindent
\par \medskip
\noindent
$\boxed{\textsc{$\gs$ of type\,\, F\,{II}}}$ This is a real form of a root system of type $\textsc{F}$.
We take the root system and basis of simple positive roots defined by 
\begin{equation*} 
\begin{cases}
 \Rad=\{{\pm}e_{i}\,{\mid}\,1{\leq}i{\leq}4\}\cup\{{\pm}e_{i}{\pm}e_{j}\,{\mid}\,1{\leq}i{<}j{\leq}4\}\cup
 \{\tfrac{1}{2}({\pm}e_{1}{\pm}e_{2}{\pm}e_{3}{\pm}e_{4})\},\\
 \Bz=\{\alpha_{1}{=}e_{1}{-}e_{2},\,\alpha_{2}{=}e_{2}{-}e_{3},\,\alpha_{3}{=}e_{3},\;\alpha_{4}{=}\tfrac{1}{2}
 (e_{4}{-}e_{1}{-}e_{2}{-}e_{3})\}
\end{cases}
\end{equation*}
and the symmetry defining $\textsc{F\,{II}}$ is 
\begin{equation*}
 \sigmaup(e_{i})=\begin{cases}
 {-}e_{i}, & i{=}1,2,3,\\
 \;\, e_{4}, & i{=}4.
 \end{cases}
\end{equation*}
The corresponding Satake diagram is 
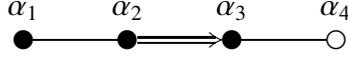
\begin{equation*}
 \xymatrix @M=0pt @R=2pt @!C=15pt{ \alpha_{1}&\alpha_{2}&\alpha_{3}&\alpha_{4}\\
  \qquad\quad 
 \\
 \medbullet\ar@{-}[r]
 &\medbullet\ar@{=>}[r]
 &\medbullet\ar@{-}[r]
 &\medcirc}
 \end{equation*}
 The fundamental CR algebras $(\gs,\qt_{\;\phiup})$ are those with $\phiup\,{\subseteq}\,\Bz_{\bullet}^{\sigmaup}\,{=}\,
 \{\alpha_{1},\alpha_{2},\alpha_{3}\}$, which are finitely nondegenerate iff $|\phiup|\,{=}\,1$. \par\smallskip
 If $\phiup\,{=}\,\phiup_{1}\,{=}\,\{\alpha_{1}\}$, then $\Qc_{\;\phiup_{1}}\,{=}\,\{\alpha\,{\in}\,\Rad\,{\mid}\,
 \xiup_{1}(\alpha)\,{\geq}\,0\},$ with 
\begin{equation*}
 \xiup_{1}(e_{i})= 
\begin{cases}
1, & i=1,4,\\
0, & i=2,3.
\end{cases}
\end{equation*}
We have $\Qc_{\;\phiup_{1}}^{n}\,{\cap}\,\Rad_{\;\bullet}^{\sigmaup}\,{=}\,\{e_{1}\}\,{\cup}\,\{e_{1}{\pm}e_{i}\,{\mid}\,i=2,3\}$.\par
Since
$({-}e_{1}{-}e_{4})\,{=}\,\sigmaup(e_{1}{-}e_{4})\,{\in}\,\sigmaup(\Qc_{\;\phiup_{1}}^{r})\,{\cap}\,\Qc^{-n}_{\;\phiup_{1}}$,
\begin{equation*} \begin{cases}
 e_{1}+({-}e_{1}{-}e_{4})={-}e_{4}\in \Qc^{-n}_{\;\phiup_{1}}\cap\sigmaup(\Qc^{-n}_{\;\phiup_{1}}),\\
 e_{1}{\pm}e_{i}+({-}e_{1}{-}e_{4})={\pm}e_{i}{-}e_{4}\in \Qc^{-n}_{\;\phiup_{1}}\cap\sigmaup(\Qc^{-n}_{\;\phiup_{1}}),\;
\text{for $i=2,3$,}
 \end{cases}
\end{equation*}
shows that $(\gs,\qt_{\;\phiup_{1}})$  has Levi order $1$.
\par\smallskip
 If $\phiup\,{=}\,\phiup_{2}\,{=}\,\{\alpha_{2}\}$, then $\Qc_{\;\phiup_{2}}\,{=}\,\{\alpha\,{\in}\,\Rad\,{\mid}\,
 \xiup_{2}(\alpha)\,{\geq}\,0\},$ with 
\begin{equation*}
 \xiup_{1}(e_{i})= 
\begin{cases}
1, & i=1,2,\\
0, & i=3,\\
2, & i=4.
\end{cases}
\end{equation*}
We have $\Qc_{\;\phiup_{1}}^{n}\,{\cap}\,\Rad_{\;\bullet}^{\sigmaup}\,{=}\,\{e_{1},e_{2},e_{1}{\pm}e_{3},e_{2}{\pm}e_{3}\}$.
\par
Since $\tfrac{1}{2}({\pm}e_{3}-e_{1}{-}e_{2}{-}e_{4})\,{=}\,\sigmaup(\tfrac{1}{2}(e_{1}{+}e_{2}{\mp}e_{3}{-}e_{4})
\,{\in}\,\sigmaup(\Qc_{\;\phiup_{2}}^{r})\,{\cap}\,\Qc^{-n}_{\;\phiup_2}$  and 
\begin{equation*} 
\begin{cases}
 e_{1}{+}\tfrac{1}{2}(e_{3}-e_{1}{-}e_{2}{-}e_{4})=\tfrac{1}{2}(e_{3}{+}e_{1}{-}e_{2}-e_{4})\in \Qc^{-n}_{\;\phiup_{2}}\cap\sigmaup(\Qc^{-n}_{\;\phiup_{2}}),\\
 e_{2}{+}\tfrac{1}{2}(e_{3}-e_{1}{-}e_{2}{-}e_{4})=\tfrac{1}{2}(e_{3}{-}e_{1}{+}e_{2}-e_{4})\in \Qc^{-n}_{\;\phiup_{2}}\cap\sigmaup(\Qc^{-n}_{\;\phiup_{2}}),\\
  (e_{1}{\pm}e_{3})+\tfrac{1}{2}({\mp}e_{3}-e_{1}{-}e_{2}{-}e_{4})=
  \tfrac{1}{2}({\pm}e_{3}{+}e_{1}{-}e_{2}-e_{4})\in \Qc^{-n}_{\;\phiup_{2}}\cap\sigmaup(\Qc^{-n}_{\;\phiup_{2}}),\\
    (e_{2}{\pm}e_{3})+\tfrac{1}{2}({\mp}e_{3}-e_{1}{-}e_{2}{-}e_{4})=
  \tfrac{1}{2}({\pm}e_{3}{-}e_{1}{+}e_{2}-e_{4})\in \Qc^{-n}_{\;\phiup_{2}}\cap\sigmaup(\Qc^{-n}_{\;\phiup_{2}}),
\end{cases}
 \end{equation*}
 $(\gs,\qt_{\;\phiup_{2}})$  has Levi order $1$.\par

%%%
\par\smallskip
 If $\phiup\,{=}\,\phiup_{3}\,{=}\,\{\alpha_{3}\}$, then $\Qc_{\;\phiup_{3}}\,{=}\,\{\alpha\,{\in}\,\Rad\,{\mid}\,
 \xiup_{3}(\alpha)\,{\geq}\,0\},$ with 
\begin{equation*}
 \xiup_{1}(e_{i})= 
\begin{cases}
1, & i=1,2,3\\
3, & i=4.
\end{cases}
\end{equation*}
We have $\Qc_{\;\phiup_{1}}^{n}\,{\cap}\,\Rad_{\;\bullet}^{\sigmaup}\,{=}\,\{e_{1},e_{2},e_{3}\}\,{\cup}\,
\{e_{i}{+}e_{j}\,{\mid}\,1{\leq}i{<}j{\leq}3\}$.
\par
Since $(-\tfrac{1}{2}(e_{1}{+}e_{2}{+}e_{3}{+}e_{4}))=\sigmaup(\tfrac{1}{2}((e_{1}{+}e_{2}{+}e_{3}{-}e_{4})))
\in\sigmaup(\Qc_{\;\phiup_{3}}^{r})\,{\cap}\,\Qc^{-n}_{\;\phiup_{3}}$ and 
\begin{equation*} 
\begin{cases}
 e_{1}{+}\tfrac{1}{2}(-e_{1}{-}e_{2}{-}e_{3}{-}e_{4})=\tfrac{1}{2}(e_{1}{-}e_{2}{-}e_{3}-e_{4})\in \Qc^{-n}_{\;\phiup_{3}}\cap\sigmaup(\Qc^{-n}_{\;\phiup_{3}}),\\
 e_{2}{+}\tfrac{1}{2}(-e_{1}{-}e_{2}{-}e_{3}{-}e_{4})=\tfrac{1}{2}({-}e_{1}{+}e_{2}{-}e_{3}-e_{4})\in \Qc^{-n}_{\;\phiup_{3}}\cap\sigmaup(\Qc^{-n}_{\;\phiup_{3}}),\\
  e_{3}{+}\tfrac{1}{2}(-e_{1}{-}e_{2}{-}e_{3}{-}e_{4})=\tfrac{1}{2}({-}e_{1}{-}e_{2}{+}e_{3}-e_{4})\in \Qc^{-n}_{\;\phiup_{3}}\cap\sigmaup(\Qc^{-n}_{\;\phiup_{3}}),\\
\begin{aligned}
(e_{i}{+}e_{j})  {+}\tfrac{1}{2}(-e_{1}{-}e_{2}{-}e_{3}{-}e_{4})=\tfrac{1}{2}(e_{i}{+}e_{j}{-}e_{\ell}{-}e_{4})
\in \Qc^{-n}_{\;\phiup_{3}}\cap\sigmaup(\Qc^{-n}_{\;\phiup_{3}})\quad \\
\text{if $\{i,j,\ell\}=\{1,2,3\}$.}
\end{aligned}
\end{cases}
 \end{equation*}
it follows $(\gs,\qt_{\;\phiup_{3}})$  has Levi order $1$.
\end{proof}
\vfill

\par\bigskip

\begin{table}[H]

  \centering
  \begin{adjustbox}{max width=1.1\textwidth}
  \begin{tabular}{|c|m{5.5em}|m{10.7em}|m{9.4em}|m{8.2em}|}
 
  \hline

  Name & $\,\,\,\,\,\,\,\,\,\,\,\,\gt^{\sigmaup}$& $\,\,\,\,\,\,\,\,\,\,\,\,\,\,\,\,\,\,\,\,\,\,\,\,\,\,\,\mathcal{S}$& 
Levi order 1& 
  Levi order 2 \\ 
   \hline
  AII&$\mathfrak{sl}(\mathbb{H},p),$    $\ell{=}2p{-}1$  & \dynkin[labels={\alpha_1,\alpha_2,\alpha_3,\alpha_{\ell-1},
\alpha_{\ell}},scale=1.8]A{II}& 
%$\emptyset$ 
 $\Bz_{\bullet}{ \supset}\phiup$
 , $|\phiup|>1.$
 &
  $\Bz_{\bullet}{ \supset}\phiup{\neq}\{\alpha_{1}\},\{\alpha_\ell\}, \,\,\,\,\,\,\, $ $|\phiup|=1.\,\,
 \,\,$  \\
  \hline
   AIIIa%
    & $\mathfrak{su}(p,q)\,\,\,$ $p{+}q{=}\ell{+}1$, 
 $2{\leq} p{\leq}\frac{\ell}{2}$
    &\dynkin[labels={\alpha_1,,\alpha_p,,,,,\alpha_{\ell-p+1},,\alpha_
    \ell},scale=2.0]A{IIIa}
    & 
 \begin{minipage}{5cm}
    $\phiup{\cap}\Bz_{\circ}{\neq}\emptyset$,
   \\
    $\epsilon_C(\phiup{\cap}\Bz_{\circ}){\cap}(\phiup{\cap}
    \Bz_{\circ}){=}\emptyset$,
\\$\epsilon_C(\phiup^{\circ}(\alpha)){\cap}\phiup{\neq}
    \emptyset$, \\$\forall\alpha{\in}\phiup{\cap}\Bz_{\circ}$,
    
    $|\phiup{\cap}\Bz_\bullet|{\leq} 2$,\\   $\textit{ if }|\phiup{\cap}\Bz_\bullet|{=}2\\ 
    \Rightarrow\phiup{\cap}\{\alpha_{p},\alpha_{\ell-p+1}\}{=}\emptyset$
\end{minipage}
& $\phiup{=}\{\alpha_i\}, \,\,\, \,\,\,\,\,$ $p {<} i {<}\ell{-}p{+}1$;
   $\phiup{=}\{\alpha_i,\alpha_j\}, $ 
   $ \,\, p {<} i{<}j{<}\ell{-}p{+}1.$  
  
\\ 
\hline
   AIIIb%
    & 
$\mathfrak{su}(p,p)$ $2{\leq}p{=}\frac{(\ell{+}1)}{2}$ 

    &\dynkin[labels={\alpha_1,,,\alpha_p,,,\alpha_\ell},scale=2.0]A{IIIb}
    & 
 \begin{minipage}{5cm}
    $\phiup{\neq}\emptyset$,  \\$\epsilon_C(\phiup){\cap}\phiup{=}
    \emptyset$,  \\$\epsilon_C(\phiup^{\circ}(\alpha)){\cap}\phiup{\neq}
    \emptyset$, \\$\forall\alpha{\in}\phiup.$
    
\end{minipage}
&$\emptyset$
\\ 
\hline
   AIV%
   &
$\mathfrak{su}(1,\ell)$  
    &\dynkin[labels={\alpha_1,,,\alpha_\ell},scale=2.0]A{IV}
    & 
 \begin{minipage}{5cm}
    $\phiup{=}\{\alpha_1\};$ $\phiup{=}\{\alpha_\ell\};\\ $
  $\phiup{=}\{\alpha_1,\alpha_i\}$,  $\phiup{=}\{\alpha_i,\alpha_\ell\}\\ $
 $1{<}i{<}\ell.$
\end{minipage}
& 
$\phiup{=}\{\alpha_i\}, \,\,\,\,\,\,\,\,\,\,\,\,\,\,\,\,\,$  $1 {<} i {<}\ell;$ 
 $\phiup{=}\{\alpha_i,\alpha_j\},\,\,\,\,\,\,\,\,\,\,\,\,\,\,\,\,\,$ $1 {<} i{<}j{<}\ell.$  \\ 
  \hline
   BI
    & $\mathfrak{so}(p,{2}\ell{+}1{-}{p})$\,\, $2{\leq} p{\leq} \ell$ &\dynkin[labels={\alpha_1,,\alpha_p,\alpha_{p+1},,\alpha_\ell},scale=1.8]B{I}
    & 
 \begin{minipage}{5cm}
    $\phiup{=}\{\alpha_{p+1}\}.$
\end{minipage}
&$
  \phiup=\{\alpha_i\},$ $p{+}1{<} i{\leq}\ell.$
  
\\ 
  \hline

  BII
   &$\mathfrak{so}(1,{2\ell})$ &\dynkin[labels={\alpha_1,\alpha_2,,\alpha_\ell},scale=2.0]B{II}
    & 
 \begin{minipage}{5cm}
   $\phiup{=}\{\alpha_{2}\}.$
\end{minipage}
& 
    $\phiup=\{\alpha_i\},\,\, 2< i \leq \ell.$  
  \\ 
  \hline

 CIIa%
    &
$\mathfrak{sp}(p,{\ell-p})$ $2p{<}\ell$     
     &\dynkin[labels={\alpha_1,,,\alpha_{2p},,,\alpha_\ell},scale=1.7]C{IIa}
    & 
 \begin{minipage}{5cm}

   $\phiup{=}\{\alpha_{2i-1}\},1{\leq} i{\leq} p;$
 $\phiup{=}\{\alpha_{2i_1{-}1},{\dots},\alpha_{2i_{k}{-}1}, \alpha_j\}$,
$k{\geq} 1, 1 {\leq} i_{1}{<}{\dots}{<}i_{k}{<}p$, 
$2i_k{-}1{<} j,\,\alpha_j{\in}\Bz_{\bullet}.$

\end{minipage}
&
$ \phiup{=}\{\alpha_j\},   j{>}2p.$\\ 
  \hline
 CIIb%
    &
$\mathfrak{sp}(p,p)$, $2p=\ell$         
     &\dynkin[labels={\alpha_1,,,,,\alpha_{\ell}},scale=2.0]C{IIb}
    & 
 \begin{minipage}{5cm}
  $
 \phiup{=}\{\alpha_{2i_1-1},{\dots}\alpha_{2i_{k}-1}\}, \\$  $ k{\geq}1, $
 $
 \,\, \,1 {\leq} i_{1}{<}{\dots}{<}i_{k}{\leq}p.$\\ 
\end{minipage}
& 

$\emptyset$\\ 
 \hline
 DIa
   &
$\mathfrak{so}(p,2\ell{-}p)$ $2\leq p\leq \ell-2$.   
   
    &\dynkin[labels={\alpha_1,,\alpha_{p},\alpha_{p+1},,\alpha_{\ell{-}1},
    \alpha_{\ell}},scale=1.8]D{Ia}
    & 
 \begin{minipage}{5cm}
   $\phiup{=}\{\alpha_{p+1}\}$. 
\end{minipage}
& 
    $
  \phiup{=}\{\alpha_i\},\,\, p{+}1 {<} i {\leq}\ell;\,\,\,\,\,\,\,\,\,\,\,\,\,\,\,\,\,\,\,$
  $\phiup{=}\{\alpha_{\ell-1},\alpha_{\ell}\}.\,\, $ \\
 
 \hline
 DIb
    &
$\mathfrak{so}(\ell{-}1,\ell{+}1)$      
     &\dynkin[labels={\alpha_1,,\alpha_{\ell-1},\alpha_{\ell}},scale=2.0]D{Ib}
    & 
 \begin{minipage}{5cm}
   $\phiup{=}\{\alpha_{\ell-1}\}$;
    $\phiup{=}\{\alpha_{\ell}\}.$
    
\end{minipage}
  & 
 \begin{minipage}{5cm}
 $\emptyset$
   \end{minipage}
\\
 \hline
 DII
    & 
$\mathfrak{so}({1,2\ell{-}1})$ 
    &\dynkin[labels={\alpha_1,,,\alpha_{\ell-1},\alpha_{\ell}},scale=2.0]D{II}
    & 
 \begin{minipage}{5cm}
   $\phiup{=}\{\alpha_{2}\}$.
\end{minipage}
& 
    $\phiup{=}\{\alpha_i\},\,\, 2 {<} i {\leq}\ell;\,\,\,\,\,\,\,\,\,\,\,\,\,\,\,\,\,\,\,\,\,\,\,\,
    $
  $\phiup{=}\{\alpha_{\ell-1},\alpha_\ell\}.\,\, \,\,\,\,\,\,\,\,\,\,\,\,\,\,\,\,$      
  \\ 

 \hline
    DIIIa  &
    $\mathfrak{su}^{*}_{2p}(\Hb)$,  $\ell{=}2p$ 
    &
\dynkin[labels={\alpha_1,,,,\alpha_{\ell-1},\alpha_{\ell}},scale=1.8]D{IIIa}    
&  $
  \phiup{=}\{\alpha_{2i_1-1},\dots,\alpha_{2i_k-1}\},$
  $
  1 {\leq} i_{1}{<}\dots{<}i_{k}{\leq}p, k{>}1.\,\, 
$   &
    $
  \phiup{=}\{\alpha_{2i-1}\},$
  $
  1 {\leq} i{\leq}p.\,\, 
$   
  
\\ 

\hline
 DIIIb&
 $\mathfrak{su}^{*}_{2p+1}(\Hb)$,  $\ell{=}2p{+}1$ 
    &\dynkin[labels={\alpha_1,,,,,\alpha_{\ell-1},
    \alpha_{\ell}},scale=2.0]D{IIIb}
    & 
 \begin{minipage}{5cm}
   $\phiup{\cap}\Bz_{\circ}{=}\{\emptyset\},\{\alpha_{\ell-1}\},\{\alpha_{\ell}\}$
   $\\\wedge$\\
   $\phiup{\cap}\Bz_{\bullet}{=}\{\alpha_{2i_1-1},\dots,\alpha_{2i_k-1}\},\,$
  $
 k{>} 1,  {\leq} i_{1}{<}\dots{<}i_{k}{\leq}p.\,\, 
 $ 
\end{minipage}
& 
    $
  \phiup{=}\{\alpha_{2i-1}\},\,$
  $
  1 {\leq} i{\leq}p.\,\, 
 $ 
\\ 

  \hline

  \hline
\end{tabular}
\end{adjustbox}
\medskip
 \caption{\label{finite nnndg} List of all finitely  nondegenerate 
  CR algebras $(\gt^{\sigmaup},\qt_{\phiup})$ of closed orbits with classical simple $\gt$.}
\end{table}

\begin{center}
\begin{table}[H]
  \begin{adjustbox}{max width=\textwidth}
\begin{tabular}{ | m{2.2em} | m{11em} | m{15em} |  m{6.8em} | } 
  \hline
  Name & $\,\,\,\,\,\,\,\,\,\,\,\,\,\,\,\,\,\,\,\,\,\,\,\,\,\,\,\,\,\mathcal{S}$&  Levi order 1&  Levi order 2\\ 
  \hline
 
  EII
    & \dynkin[labels={\alpha_1,\alpha_2,\alpha_3,\alpha_4,\alpha_5,
    \alpha_6},scale=2.0]E{II}
	& 
   $\phiup{=}\{\alpha_1\},\{\alpha_3\},\{\alpha_5\},\{\alpha_6\},$ 
    $\{\alpha_1,\alpha_5\},\{\alpha_3,\alpha_6\}$.
& 
    $\emptyset$    
    
      \\ 
    \hline

 EIII 
    &\dynkin[labels={\alpha_1,\alpha_6,\alpha_2,\alpha_3,\alpha_4,
    \alpha_5},scale=2.0]E{III}
    & 
 \begin{minipage}{6.5cm}
 $\phiup{=}\{\alpha_1\}, $  $\{\alpha_1,\alpha_2\}, $ 
    $\{\alpha_1,\alpha_3\},$ $ \{\alpha_1, \alpha_4\}, $  
    $\{\alpha_1,\alpha_2,\alpha_4\}, $ $\{\alpha_1,\alpha_3,\alpha_4\},$  
      $ \{\alpha_5\},  $  $ \{\alpha_5,\alpha_2\}, $
       $ \{\alpha_5,\alpha_4\}, $   $\{\alpha_5,\alpha_3\}, $ 
       $ \{\alpha_5,\alpha_2,\alpha_4\}, $  $ \{\alpha_5,\alpha_2,\alpha_3\},$
       $\{\alpha_2\},\{\alpha_3\},\{\alpha_4,\}.$
    $\{\alpha_2,\alpha_3\},$ $\{\alpha_2,\alpha_4\},$
    $\{\alpha_3,\alpha_4\},$ $\{\alpha_2,\alpha_3,\alpha_4\}.$
    \end{minipage}
& 
   $\emptyset$
\\ 
  \hline
  EIV 
    & \dynkin[labels={\alpha_1,\alpha_4,\alpha_2,\alpha_3,\alpha_5,
    \alpha_6},scale=2.0]E{IV}
& 
$\phiup{\subseteq}\Rad_{\bullet}$,
$|\phiup|{\leq}2$,
   $\phiup\,{\neq}\{\alpha_3\},\{\alpha_{3},\alpha_{4}\}.$ 
& 
    $\phiup=\{\alpha_3\}.$

    \\ 
  \hline
  EVI 
    & \dynkin[labels={\alpha_7,\alpha_5,\alpha_6,\alpha_4,\alpha_3,
    \alpha_2, \alpha_1},scale=2.0]E{VI}
& 
   $\phiup{\subseteq}\Rad_{\bullet},  |\phiup|=2,3.$
& 
    $\phiup{=}\{\alpha_1\},\{\alpha_3\},\{\alpha_5\}.$    
    
      \\ 
    \hline
   EVII   &\dynkin[labels={\alpha_7,\alpha_5,\alpha_6,\alpha_4,
  \alpha_3,\alpha_2,\alpha_1},scale=2.0]E{VII}
& 
$\phiup{\subseteq}\Rad_{\bullet}$,
$|\phiup|{\leq}2$,
   $\phiup\,{\neq}\{\alpha_4\},\{\alpha_{4},\alpha_{5}\}.$ 
& 
    $\phiup=\{\alpha_4\}.$

       \\ 
    \hline
    EIX  &
\dynkin[labels={\alpha_8,\alpha_6,\alpha_7,\alpha_5,\alpha_4,\alpha_3,
\alpha_2,\alpha_1},scale=2.0]E{IX}     
& 
$\phiup{\subseteq}\Rad_{\bullet}$,
$|\phiup|{\leq}2$,
   $\phiup\,{\neq}\{\alpha_5\},\{\alpha_{5},\alpha_{6}\}.$ 
& 
    $\phiup=\{\alpha_5\}.$          
 
    \\
    \hline
    \medskip

   FII& 
\dynkin[labels={\alpha_1,\alpha_2,\alpha_3,\alpha_4},scale=2.0]F{II}    
&$\phiup=\{\alpha_{1}\},\{\alpha_{2}\},\{\alpha_{3}\}.$
& $\emptyset$
 \\ 
  \hline
  
\end{tabular}
\end{adjustbox}
\smallskip
  \caption{\label{finite nnndg ex} List of all finitely  nondegenerate 
  CR   algebras $(\gt^{\sigmaup},\qt_{\phiup})$ of closed orbits with exceptional simple $\gt$.}
\end{table}
\end{center}

\subsection*{Further remarks: the complex type cases}
In Tables \ref{finite nnndg}, \ref{finite nnndg ex}  we considered simple real Lie algebras 
$\gt^{\sigmaup}$ having a simple
complexification $\gt$ (\textit{real type}). 
Our approach  
generalizes these approach to \emph{complex type} simple real algebras,
i.e. simple complex Lie algebras considered as real Lie algebras by restriction of the field of scalars. 
The complexification $\gt$ of a $\gt^{\sigmaup}$ of the complex type is the direct sum
of two compies of $\gs$. 
Having fixed any real form of the underlying complex Lie algebra $\gs_{\C}$ 
of $\gs$,
an 
involution 
$\sigmaup$ defining the real form $\gt^{\sigmaup}$ can be defined by 
using the corresponding conjugation of $\gs_{\C}$ and setting
\begin{equation}
 \gt\simeq\gs_{\C}\oplus\gs_{\C}\ni(Z,W)\longrightarrow (\bar{W},\bar{Z})\in\gs_{\C}\oplus\gs_{\C}\simeq\gt.
\end{equation}

Since the root system of $\gt$ does not contain imaginary roots, by
the criteria of Lemma~ \ref{p5.8} the parabolic CR algebras $(\gs,\qt_{\phiup})$ with $\gs$ of the complex type
are either holomorphically degenerate or $1$-nondegenerate.\par
A criterion of finitely nondegeneracy was given in
\cite[Thm.11.5]{AlMeNa2006} in terms of cross-marked Satake diagrams. 
Those corresponding to minimal orbits consist of two copies of the Dynkin diagram of
$\gs_{\C}$, with homologous simple roots joined by a curved arrow. We denote by $\epsilonup$
the symmetry exchanging homologous roots of $\Bz$. \par
As above, the minimal parabolic CR algebras have the form $(\gs,\qt_{\phiup})$ for a subset $\phiup$
of the set $\Bz$ of positive simple roots.  \par
If $\alpha\,{\in}\,\phiup$, we denote by
${\phiup}^{\circ}(\alpha)$ the connected component of $\alpha$ in $(\Bz{\setminus}\Phi){\cup}\{\alpha\}$.
\begin{prop}
 Let $(\gs,\qt_{\phiup})$ be a minimal parabolic CR algebra, with a simple $\gs$ of the complex type. The following
 are equivalent 
\begin{enumerate}
 \item $(\gs,\qt_{\phiup})$ is finitely nondegenerate;
 \item $(\gs,\qt_{\phiup})$ is Levi $1$-nondegenerate;
 \item $\forall\alpha\,{\in}\,\phiup$, we have 
$
 \epsilonup({\phiup}^{\circ}(\alpha))\cap\phiup\neq\emptyset.
$ \qed
\end{enumerate}
\end{prop}
\begin{exam} \label{ex.2}
 The $CR$ algebra  $(\gt^{\sigmaup},\qt_{\phiup})$
 described by the cross-marked Satake diagram 
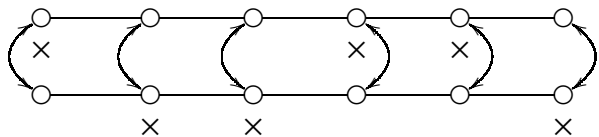
\begin{equation*} 
    \xymatrix@R=-.15pc{
     \!\!\medcirc\!\!\ar@{-}[r]\ar@{<->}@/_1pc/[ddd]
&\!\!\medcirc\!\! \ar@{-}[r]\ar@{<->}@/_1pc/[ddd]
&\!\!\medcirc\!\!\ar@{<->}@/_1pc/[ddd]\ar@{-}[r]
&\!\!\medcirc\!\!\ar@{<->}@/^1pc/[ddd]\ar@{-}[r]
&\!\!\medcirc\!\! \ar@{-}[r]\ar@{<->}@/^1pc/[ddd]
&\!\!\medcirc\!\! \ar@{<->}@/^1pc/[ddd]\\
\times&&&\times&\times\\
\quad\\
\!\!\medcirc\!\!\ar@{-}[r]
&\!\!\medcirc\!\! \ar@{-}[r]
&\!\!\medcirc\!\!\ar@{-}[r]
&\!\!\medcirc\!\!\ar@{-}[r]
&\!\!\medcirc\!\!\ar@{-}[r]
&\!\!\medcirc\!\! \\
  &\times&\times&&&\times}
\end{equation*}
 with $\gt^{\sigmaup}\,{\simeq}\,\slt_{7}(\C),$ 
is fundamental and $1$-nondegenerate by  criteria of $(2)$ Lemma~\ref{p5.8}.
\end{exam} 
\appendix\section{Elemetary conjugation diagrams}
 We conclude presenting a table,  taken from explicit computations contained in \cite[pp.18--21]{Ar1962},  that summarize all roots conjugation rules in all different subgraph types.
\begin{center}
\begin{table}[H]
  \begin{adjustbox}{max width=0.9\textwidth}
\begin{tabular}{ | m{2.2em} | m{10em} | m{15.2em} |  } 
\hline
 Type & $\,\,\,\,\,\,\,\,\,\,\,\,\,\,\,\,\,\,\,\,\,\,\,\,\,\,\mathcal{S}$ &$\,\,\,\,\,\,\,\,\,\,\,\,\,\,\,\,\,\,\,\,\,\,\,\,\,\,\,\,\,\,\,\,\,\sigmaup$\\ 
  \hline
 $A_1$ &  \,\,\,\,\,\,\,\,\,\,\,\,\,\,\,\,\,\,\,\,\,\, \dynkin[labels={\alpha_1},scale=2.0] A{o} &$\sigmaup (\alpha_1)=\alpha_1$\\
  \hline
  $A_1{\times} A_1$ & 
  \,\,\,\,\,\,\,\,\,\,\,
  \dynkin[labels={\alpha_1},scale=2.0] A{o}
 \tikz [<->] \draw (0,0) -- (1,0);
   \dynkin[labels={\alpha_2},scale=2.0] A{o}

  &$\sigmaup (\alpha_1)=\alpha_2$\\ 
  \hline
   $A_3$ & \,\,\,\,\,\,\,\,\,\,\,\,\,\,\dynkin[labels={\alpha_1,\alpha_2,\alpha_3},scale=2.0] A{*o*}
   &$\sigmaup (\alpha_2)=\alpha_1+\alpha_2+\alpha_3$\\ 
  \hline
   $A_{\ell}$ &

\dynkin[%
scale=2.0,
labels*={\alpha_1,\alpha_2,\alpha_3,\alpha_{\ell-2},\alpha_{\ell-1},\alpha_{\ell}},
involutions={16}]A{o**.**o}
   
   &$\sigmaup (\alpha_1)=\alpha_2+\dots+\alpha_{\ell}\,\,\,\,\,\,\,\,\,\,\,\,\,\,\,\,\,\,\,\,$
   
   \\ 
  \hline
   $B_{\ell}$ &
\dynkin[%
scale=2.0,
labels*={\alpha_1,\alpha_2,\alpha_3,\alpha_{\ell-2},\alpha_{\ell-1},\alpha_{\ell}}]B{o**.***}   
   
    &$\sigmaup (\alpha_1){=}\alpha_1{+}2(\alpha_2{+}\dots{+}\alpha_{\ell})$\\ 
  \hline 
  $C_{\ell}$ & 
\dynkin[%
scale=2.0,
labels*={\alpha_1,\alpha_2,\alpha_3,\alpha_{\ell-2},\alpha_{\ell-1},\alpha_{\ell}}]C{*o*.***}     
  
  &$\sigmaup (\alpha_2){=}\alpha_1{+}\alpha_2{+}2(\alpha_3{+}\dots{+}\alpha_{\ell-1})$ ${+}\alpha_{\ell}$\\ 
  \hline 
  $D_{\ell}$ & 
\dynkin[%
scale=2.0,
labels*={\alpha_1,\alpha_2,\alpha_3,\alpha_{\ell-2},\alpha_{\ell-1},\alpha_{\ell}}]D{o**.***}      
  &$\sigmaup (\alpha_1){=}\alpha_1{+}2(\alpha_3{+}\dots{+}\alpha_{\ell-2})$ ${+}\alpha_{\ell-1}{+}\alpha_{\ell}$\\ 
  \hline
   $F_4$ & 
    \,\,\,\,\,\,\,\,\,\,\,\,
\dynkin[%
scale=2.0,
labels*={\alpha_1,\alpha_2,\alpha_3,\alpha_4}]F{***o}         
   
   &$\sigmaup (\alpha_4)=\alpha_1+2\alpha_2+3\alpha_3+\alpha_4$\\ 
  \hline
  
\end{tabular}
\end{adjustbox}
\smallskip
  \caption{\label{Araki} Elementary conjugation diagrams.}
\end{table}
\end{center}

\bibliographystyle{amsplain}
\renewcommand{\MR}[1]{}
\bibliography{homog}

\end{document}